\documentclass[leqno]{amsart}
\usepackage{amssymb}
\usepackage{amsmath, amsfonts,enumerate}
\usepackage{hyperref}
\usepackage{amsthm}
\usepackage{latexsym,bm}
\usepackage{euscript}

\let\cal\mathcal
\makeatletter \@mparswitchfalse \makeatother
\newtheorem{theorem}{Theorem}
\newtheorem{lemma}[theorem]{Lemma}

\newtheorem{corollary}[theorem]{Corollary}
\newtheorem{proposition}[theorem]{Proposition}
\theoremstyle{remark}
\newtheorem{remark}[theorem]{Remark}

\theoremstyle{definition}
\newtheorem{definition}[theorem]{Definition}

\theoremstyle{remark}
\newtheorem{example}[theorem]{Example}

\numberwithin{equation}{section}
\numberwithin{theorem}{section}

\def\M{\cal{M}}
\def\H{\cal{H}}
\let\<\langle
\def\ch{\raise 0.5ex \hbox{$\chi$}}
\def\T{\tau}
\def\E{\cal{E}}

\let\\\cr
\let\phi\varphi

\let\epsilon\varepsilon

\def\Re{\operatorname{Re}}

\def\log{\operatorname{log}}
\renewcommand{\a}{\alpha}
\renewcommand{\b}{\beta}

\newcommand{\N}{\cal{N}}
\newcommand{\h}{\mathsf{h}}

\newcommand{\tr}{\mbox{\rm tr}}

\begin{document}

\title[Davis decomposition]{Noncommutative Davis type decompositions and applications}
\author[ Randrianantoanina]{Narcisse Randrianantoanina}
\address{Department of Mathematics, Miami University, Oxford,
Ohio 45056, USA}
 \email{randrin@miamioh.edu}
 
 \author[Wu]{ Lian Wu}
\address{School of Mathematics  and Statistics, Central South University, Changsha 410075, China}
\email{wulian@cnu.edu.cn}
%\thanks{Wu is supported by NSFC (11601526) and the China Postdoctoral %Foundation (2016M602420, 2017T100606)} 

\author[Xu]{Quanhua Xu}
\address {Institute for Advanced Study in Mathematics, Harbin Institute of Technology,  Harbin 150001, China; and Laboratoire de Math{\'e}matiques, Universit{\'e} de Bourgogne Franche-Comt{\'e}, 25030 Besan\c{c}on Cedex, France; and Institut Universitaire de France}
\email{qxu@univ-fcomte.fr}
 
 \date{\today}


%%%%%%%%%%%%%%%%%%%%%%%%%%%%%%%%%%%%%%%%%%%%%%%%%%%%%%%%%%%%%%%%%%%%%%%

\begin{abstract} We prove the noncommutative  Davis decomposition for the column Hardy space  $\H_p^c$ for all $0<p\leq 1$. A new feature of our Davis decomposition is a simultaneous control of $\H_1^c$ and $\H_q^c$ norms for any  noncommutative martingale in $\H_1^c \cap \H_q^c$ when $q\geq 2$. As applications, we   show that  the Burkholder/Rosenthal  inequality holds for  bounded martingales in a noncommutative symmetric space  associated with a function space $E$ that  is either an interpolation 
  of  the couple $(L_p, L_2)$  for some $1<p<2$  or  is an interpolation of the couple $(L_2, L_q)$  for some $2<q<\infty$.  
  We  also obtain the corresponding $\Phi$-moment Burkholder/Rosenthal   inequality for Orlicz functions  that are either $p$-convex and $2$-concave for some $1<p<2$ or  are $2$-convex and $q$-concave for some $2<q<\infty$.
\end{abstract}

%%%%%%%%%%%%%%%%%%%%%%%%%%%%%%%%%%%%%%%%%%%%%%%%%%%%%%%%%%%%%%%%%%%%%%%

\maketitle

%%%%%%%%-----------------------------

%\setcounter{section}{-1}

%%%%%%%%%%%%%%%%%%%%%%%%%%%%%%%%%%%%%%%%%%%%%%%%%%%%%%%%%%%%%%%%%%%%%%%
%%%%%%%%%%%%%%%%%%%%%%%%%%%%%%%%%%%%%%%%%%%%%%%%%%%%%%%%%%%%%%%%%%%%%%%

\section{Introduction}

%%%%%%%%%%%%%%%%%%%%%%%%%%%%%%%%%%%%%%%%%%%%%%%%%%%%%%%%%%%%%%%%%%%%%%%
%%%%%%%%%%%%%%%%%%%%%%%%%%%%%%%%%%%%%%%%%%%%%%%%%%%%%%%%%%%%%%%%%%%%%%%

This paper follows the current line of investigation on noncommutative martingale inequalities. Many classical results have been generalized  to the noncommutative setting. One of them, directly relevant to the subject of the present paper is the so called Davis decomposition (\cite{Da2}).  The original Davis  decomposition is fundamental in classical martingale theory and has been generalized to various contexts. For instance, the vector-valued case is nowadays well-known in the literature, a version of the  Davis decomposition  for a special class of martingales called Hardy martingales was studied recently in  \cite{Muller2}.

Recall that  for the noncommutative setting,  the Davis decomposition for the noncommutative  martingale Hardy spaces $\H_1$  was obtained in \cite{Perrin} using duality arguments.  A constructive approach appeared in \cite{Junge-Perrin}  for the  space $\H_p$ for  $1\leq p<2$. The noncommutative Davis decomposition  has proven to  be a powerful tool in noncommutative martingale inequalities; for instance, it plays a prominent role in establishing various forms of Doob maximal inequalities in \cite{Hong-Junge-Parcet} as well as in  the study of continuous time noncommutative martingale inequalities in \cite{Junge-Perrin}.

It is our intention in this paper to investigate the case $0<p\leq 1$. We provide a Davis type decomposition for certain class of  sequences in the column-$L_p$-spaces. This can be roughly described as splitting any adapted sequence in the column-$L_p$-space into a diagonal part and an adapted sequence that belongs to the corresponding conditioned column-$L_p$-space.
Even for the commutative case, our result for $0<p<1$  do not seem to be available in the literature.  An important  new feature of our Davis decomposition is that when applied to martingales, it gives a simultaneous control of the column Hardy spaces $\H_1^c$ and $\H_q^c$ norms when $q\geq 2$. More precisely, for any given $q\geq 2$,  any  martingale $x$ in  the intersection of Hardy spaces $\H_1^c\cap \H_q^c$ can be written as a sum of two martingales $y$ and $z$ such that  $\big\| y\big\|_{\h_p^d} + \big\|z \big\|_{\h_p^c} \leq C\big\|x\big\|_{\H_p^c}$ for all $1 \leq p\leq q$ where $\h_p^d$ and $\h_p^c$ denote the diagonal Hardy space and the column conditioned Hardy space respectively. Decompositions with such simultaneous control of norms are very useful for the study of noncommutative martingales in continuous time (see \cite{Junge-Perrin}). They are also very essential in the study of martingale  Hardy spaces associated to noncommutative symmetric spaces (\cite{RW, RW2}). Indeed, our primary motivation comes from the latter. The simultaneous nature  of our Davis decomposition allows us to extend it through the use of interpolation to martingales in  certain  noncommutative symmetric spaces satisfying some natural conditions. This in turn provides    a general framework to systematically transfer  results involving square functions which are generally referred  to as 
Burkholder-Gundy inequality to  combinations of conditioned square functions and diagonal parts known as Burkholder/Rosenthal  inequality.  

Recall the noncommutative Burkholder/Rosenthal inequalities from  \cite{JX}. It asserted that if  $2\leq p <\infty$ and $x=(x_n)_{n\geq 1}$ is a noncommutative martingale that is $L_p$-bounded  then
\begin{equation}\label{nc1}
\big\|x\big\|_p \simeq_p \max\Big\{ \big\|s_c(x)\big\|_p,  \big\|s_r(x)\big\|_p, \big( \sum_{n\geq 1} \big\|dx_n\big\|_p^p \big)^{1/p}\Big\},
\end{equation}
where $s_c(x)$ and $s_r(x)$ denote the column and row versions  of  conditioned square functions which we refer to the next section for formal definitions. The corresponding inequalities for  the range $1<p<2$   dual  to \eqref{nc1} reads as follows: if $x=(x_n)_{n\geq 1}$ is a noncommutative martingale in $L_2(\M)$ then 
\begin{equation}\label{nc2}
\big\|x \big\|_p \simeq_p \inf\Big\{ \big\|s_c(y)\big\|_p +  \big\|s_r(z)\big\|_p + \big( \sum_{n\geq 1} \big\|dw_n\big\|_p^p \big)^{1/p}\Big\},
\end{equation}
where  the infimum is taken over all $x=y+z +w$  with $y$, $z$, and $w$ martingales. The differences between the two cases $1<p<2$ and $2\leq p<\infty$ are now well-understood in the field. The natural next step is to classify  noncommutative symmetric spaces for which  either \eqref{nc1} or \eqref{nc2} remains valid. Naturally, interpolation plays a significant role in this line of research. It was established in \cite{Dirksen2} that if a  function space $E$ is an interpolation space of  the couple $(L_p,L_q)$ for $2<p<q<\infty$, then  
\begin{equation}\label{nc3}
\big\|x\big\|_{E(\M)}\simeq_E \max\Big\{ \big\|s_c(x)\big\|_{E(\M)},  \big\|s_r(x)\big\|_{E(\M)},  \big\| (dx_n)_{n\geq 1}\big\|_{E(\M \overline{\otimes} \ell_\infty)}\Big\}.
\end{equation}
On the other hand, the dual result was proved in \cite{RW} which states that if $E$
is a symmetric space  that is an interpolation space of  the couple $(L_p,L_q)$ for $1<p<q<2$ then 
\begin{equation}\label{nc4}
\big\|x \big\|_{E(\M)} \simeq_E \inf\Big\{ \big\|s_c(y)\big\|_{E(\M)} +  \big\|s_r(z)\big\|_{E(\M)} +   \big\| (dw_n)_{n\geq 1}\big\|_{E(\M \overline{\otimes} \ell_\infty)} \Big\},
\end{equation}
where as in \eqref{nc2}, the  infimum is taken over all  decompositions $x=y+z +w$  with $y$, $z$, and $w$ martingales in $E(\M,\T)$.
The situation at the endpoints were left open in  \cite{RW}. We solve this problem positively. More precisely, we obtain that \eqref{nc3} and \eqref{nc4} remain valid for  $E$ being an interpolation of the couple $(L_2, L_q)$ for $2<q<\infty$, respectively, $(L_p,L_2)$  for $1<p<2$. As noted earlier, our new Davis decomposition provides the decisive ingredient in our argument. 

In the last part of the paper, we consider  the noncommutative Burkholder/Rosental  inequalities using moments associated with Orlicz spaces. These moments are generally referred   to in the literature  as $\Phi$-moment inequalities. For the classical setting, this topic goes back to \cite{Bu1, Burkholder-Davis-Gundy}.  For noncommutative martingales, this line of research was initiated by Bekjan and Chen in \cite{Bekjan-Chen} where they provided   several $\Phi$-moment  inequalities such as $\Phi$-moment versions of the noncommutative Khintchine inequalities and noncommutative Burkholder-Gundy inequality among other closely related results. Subsequently, $\Phi$-moment analogues of other inequalities were also 
considered (see for instance, \cite{Bekjan-Chen-Ose, Dirksen, Dirksen-Ricard}). 
Recently, the  sharpest result for the $\Phi$-moment analogue of the noncommutative Burkholder-Gundy inequalities  was obtained by Jiao {\it et al.} (see \cite[Theorem~7.2]{Jiao-Sukochev-Zanin-Zhou}). Using our general approach, we extend their result to the $\Phi$-moment analogues of the noncommutative Burkholder inequalities. More precisely, if the Orlicz  function is $p$-convex and $2$-concave (for some $1<p<2$), respectively $2$-convex and $q$-concave (for some $2<q<\infty$), then the $\Phi$-moment analogue of \eqref{nc3}, respectively \eqref{nc4}, holds. Our results in this part solve some problems left open in \cite{RW2}.

The paper is organized as follows. In the next section, we collect notions and notation from noncommutative symmetric spaces and noncommutative martingales necessary for the whole paper. Section~3 is devoted to the statements and proofs of our version of noncommutative Davis decompositions for the full range $0<p\leq 1$ (Theorem~\ref{infinite1} and Theorem~\ref{infinite2}) along with some immediate corollaries.  We also  provide  in this section an extension  of the Davis decomposition to the case of noncommutative symmetric spaces (Theorem~\ref{E-davis}). In the last section, we  give the main applications in the forms of various Burkholder/Rosenthal  inequalities for martingales in noncommutative symmetric spaces and  their modular versions.

%%%%%%%%%%%%%%%%%%%%%%%%%%%%%%%%%%%%%%%%%%%%%%%%%%%%%%%%%%%%%%%%%%%%%%%
%%%%%%%%%%%%%%%%%%%%%%%%%%%%%%%%%%%%%%%%%%%%%%%%%%%%%%%%%%%%%%%%%%%%%%%

\section{Preliminaries}

%%%%%%%%%%%%%%%%%%%%%%%%%%%%%%%%%%%%%%%%%%%%%%%%%%%%%%%%%%%%%%%%%%%%%%%
%%%%%%%%%%%%%%%%%%%%%%%%%%%%%%%%%%%%%%%%%%%%%%%%%%%%%%%%%%%%%%%%%%%%%%%

%%%%%%%%%%%%%%%%%%%%%%%%%%%%%%%%%%%%%%%%%%%%%%%%%%%%%%%%%%%%%%%%%%%%%%%
\subsection{Noncommutative symmetric spaces}
%%%%%%%%%%%%%%%%%%%%%%%%%%%%%%%%%%%%%%%%%%%%%%%%%%%%%%%%%%%%%%%%%%%%%%%

Throughout this paper, $\M$ will always denote a semifinite von Neumann algebra equipped with a faithful normal semifinite trace $\T$.   $L_0(\M,\T)$ denotes  the associated topological $*$-algebra of measurable operators  and $\mu(x)$  the  generalized singular number of an element  $x \in L_0(\M,\T)$. If $\M$ is the abelian von Neumann algebra $L_\infty(0,\infty)$ with  the trace given by  integration  with respect to Lebesgue  measure, $L_0(\M,\T)$  becomes  the space of  those  measurable  complex functions  on $(0,\infty)$ which are bounded except on a set of finite measure and for  $f\in L_0(\M,\T)$,  $\mu(f)$ is the usual decreasing rearrangement of  $f$. We refer to \cite{PX} for more information on noncommutative integration.

A Banach function space  $(E,\|\cdot\|_E)$ of measurable functions  on the interval $(0,\infty)$ is called \emph{symmetric} if for any $g \in E$ and any $f \in L_0(0,\infty)$ with $\mu(f) \leq \mu(g)$, we have $f \in E$ and $\|f\|_E \leq \|g\|_E$.
For such a space  $E$, we define the  corresponding  noncommutative space by setting:
\begin{equation*}
E(\M, \T) = \big\{ x \in
L_0(\M,\T)\ : \ \mu(x) \in E \big\}. 
\end{equation*}
Equipped with the norm
$\|x\|_{E(\M,\T)} := \| \mu(x)\|_E$,   $E(\M,\T)$ becomes a complex Banach space (\cite{Kalton-Sukochev,X}) and is usually referred to as the \emph{noncommutative symmetric space} associated with $\M$ and  $E$. An extensive discussion of the various properties of such spaces can be found in \cite{DDP1,DDP4, DDP3,PX3,X}.
 We remark that  if $1\leq p<\infty$ and $E=L_p(0, \infty)$, then $E(\M, \T)=L_p(\M,\T)$  where $L_p(\M,\T)$ is  the usual noncommutative $L_p$-space associated with  $(\M,\T)$. 

In this paper, we will only consider  symmetric spaces that are interpolations of the couple $(L_p, L_q)$ for $1\leq p<q\leq \infty$. 
For a given compatible  Banach couple  $(X, Y)$, we recall that a Banach space $Z$ is called  an \emph{interpolation space} if $X \cap Y \subseteq Z \subseteq X +Y$ and whenever  a bounded linear operator $T: X + Y \to X + Y$ is such that $T(X) \subseteq X$ and $T(Y) \subseteq Y$, we have $T(Z)\subseteq Z$
and $\|T:Z\to Z\|\leq C\max\{\|T :X \to X\| ,\; \|T: Y \to Y\|\}$ for some constant $C$. 
In this case, we write $Z\in {\rm Int}(X,Y)$.
  We refer to \cite{BENSHA,BL,KaltonSMS} for more unexplained definitions and terminology from interpolation. 
  We record here two facts   that  we will use repeatedly. The first is the fact that interpolation lifts to noncommutative symmetric spaces. More precisely, we have:

\begin{lemma}[\cite{PX3}]\label{Operator-interpolation}  Let $1\leq p<q\leq \infty$. Assume that $E \in {\rm Int}(L_p, L_q)$ and  $\M$ and $\N$ are semifinite von Neumann algebras. Let $T: L_p(\M) + L_q(\M) \to L_p(\N) +L_q(\N)$ be a linear operator such that $T: L_p(\M) \to L_p(\N)$ and $T: L_q(\M) \to L_q(\N)$ are bounded. Then $T$ maps $E(\M)$ into $E(\N)$ and the resulting operator $T: E(\M) \to E(\N)$ is bounded and satisfies
\[
\big\| T:E(\M) \to E(\N)\big\| \leq C \max\big\{ \big\| T: L_p(\M) \to L_p(\N)\big\|, \big\| T: L_q(\M) \to L_q(\N)\big\| \big\},
\]
where $C$ is the interpolation constant of $E$ relative to the couple $(L_p, L_q)$.
 \end{lemma}
 
The second is the fact that  any function space $E\in {\rm Int}(L_p,L_q)$ can be described by a
 concrete interpolation method  involving  the notions of $K$-functionals and $J$-functionals. We only describe here a version that we need. First, we recall that for a compatible couple $(X,Y)$, the $J$-functional of $z \in X \cap Y$ is given by 
 \[
 J(x,t;X,Y)=\max\big\{\big\|z\big\|_X, t\big\|z\big\|_Y\big\}, \quad t>0.
 \]
The dual functional called $K$-functional  of $z \in X +Y$ is given by
\[
K(z,t; X,Y) =\inf\big\{ \big\|x\big\|_X + t\big\|y\big\|_Y : z=x +y\big\}, \quad t>0.
\] 
 Fix  $(X, Y)$   and  a symmetric Banach function space $F$ on $(0,\infty)$. For $x \in X +Y$, let  $x=
\sum_{\nu \in \mathbb{Z}} u_\nu$ be a (discrete) representation of $x$ and set:
\[
\underline{j}\big(\{u_\nu\}_{\nu},t\big) =\sum_{\gamma\geq \nu+1} 2^{-\gamma} J(u_\gamma, 2^\gamma) \quad \text{for} \  t\in [2^\nu, 2^{\nu+1}).
\]
We define  the interpolation space $(X, Y)_{F, \underline{j}}$ to be the space of elements $x\in X +Y$  such that
\[
\big\|x\big\|_{F, \underline{j}} :=\inf\Big\{ \Big\| \underline{j}(\{u_\nu\}_\nu, \cdot)\Big\|_F\Big\} <\infty
\]
with the infimum being  taken over all representations of $x$ as above. By combining results  of Brudnyi and Krugliak (see \cite[Theorem~6.3]{KaltonSMS}),  \cite{Bennett1}, and \cite[Corollary~2.2]{PX3}, we derive the following general result:

\begin{lemma}\label{interpolation1}
Let  $1\leq p<q \leq \infty$ and  $E$  be a  symmetric Banach  function space on $(0,\infty)$ with $E\in {\rm Int}(L_p, L_q)$.
 There exists a symmetric Banach  function space $F$ on $(0,\infty)$    so that for every semifinite  von Neumann algebra $(\N,\sigma)$,
\[
E(\N)=\big(L_p(\N), L_q(\N)\big)_{F,\underline{j}},
\]
with  equivalent norms depending only on $E$, $p$, and $q$.
\end{lemma}

%%%%%%%%%%%%%%%%%%%%%%%%%%%%%%%%%%%%%%%%%%%%%%%%%%%%
\subsection{Martingales and Hardy spaces}
%%%%%%%%%%%%%%%%%%%%%%%%%%%%%%%%%%%%%%%%%%%%%%%%%%%%

We now briefly describe the general setup for  martingales in noncommutative symmetric spaces.
  Denote by $(\M_n)_{n \geq 1}$ an
increasing sequence of von Neumann subalgebras of ${\M}$
whose union  is weak*-dense in
$\M$. For $n\geq 1$, we assume that there exists a trace preserving conditional expectation ${\E}_n$ 
from ${\M}$ onto  ${\M}_n$.  It is well-known that if  $\T_n$  denotes the restriction of $\T$ on $\M_n$, then $\E_n$ extends to a contractive projection from $L_p(\M,\T)$ onto $L_p(\M_n, \T_n)$ for all $1\leq p \leq \infty$. More generally, if $E$ is a symmetric Banach function space on $(0,\infty)$  that belongs to ${\rm Int}(L_1, L_\infty)$, then  for every $n\geq 1$,   $\E_n$ is bounded  from $E(\M,\T)$ onto $E(\M_n,\T_n)$.  

\begin{definition}
A sequence $x = (x_n)_{n\geq 1}$ in $L_1(\M)+\M$ is called \emph{a
noncommutative martingale} with respect to $({\M}_n)_{n \geq
1}$ if $\mathcal{E}_n (x_{n+1}) = x_n$ for every $n \geq 1.$
\end{definition}
If in addition, all $x_n$'s belong to $E(\M)$  then  $x$ is called an $E(\M)$-martingale.
In this case, we set
\begin{equation*}\| x \|_{E(\M)} = \sup_{n \geq 1} \|
x_n \|_{E(\M)}.
\end{equation*}
If $\| x \|_{E(\M)} < \infty$, then $x$ is called
a bounded $E(\M)$-martingale. 

Let $x = (x_n)$ be a noncommutative martingale with respect to
$(\M_n)_{n \geq 1}$.  Define $dx_n = x_n - x_{n-1}$ for $n
\geq 1$ with the usual convention that $x_0 =0$. The sequence $dx =
(dx_n)$ is called the \emph{ martingale difference sequence} of $x$. A martingale
$x$ is called  \emph{a finite martingale} if there exists $N$ such
that $d x_n = 0$ for all $n \geq N.$
In the sequel, for any operator $x\in L_1(\M)+\M$,  we denote $x_n=\E_n(x)$ for $n\geq 1$.
 We observe that conversely,  if $E\in {\rm Int}(L_p,L_q)$ for $1<p\leq q <\infty$ and  satisfies the Fatou property, then any  bounded $E(\M)$-martingale $x=(x_n)_{n\geq 1}$ is of the form $(\E_n(x_\infty))_{n\geq 1}$ where  $x_\infty \in E(\M)$ satisfying $\|x\|_{E(\M)} \approx_E \|x_\infty\|_{E(\M)}$, with equality if $E$ is an exact interpolation space.

Let us now  review the definitions of the square functions and Hardy spaces of noncommutative martingales.
Following \cite{PX}, we define  the following column square functions
relative to a  martingale $x = (x_n)$:
 \[
 S_{c,n} (x) = \Big ( \sum^n_{k = 1} |dx_k |^2 \Big )^{1/2}, \quad
 S_c (x) = \Big ( \sum^{\infty}_{k = 1} |dx_k |^2 \Big )^{1/2}\,.
 \]
For $0 <p < \infty$, the column  martingale Hardy space 
 $\mathcal{H}_p^c (\mathcal{M})$ is defined to be the space of all martingales $x$ for which $S_c(x)$  belongs to $L_p(\M,\T)$.  More generally, if $E$ is a symmetric Banach function space,
we define  $\mathcal{H}_E^c (\mathcal{M})$ to  be the space of all martingales $x=(x_n)_{n\geq 1}$ in $E(\M)$ for which  $S_c(x)$ belongs to $E(\M)$.   $\mathcal{H}_E^c (\mathcal{M})$  becomes  a Banach space when equipped with the norm
\[
\| x \|_{\mathcal{H}_E^c}= \| S_c(x)\|_{E(\M)}\,. \]

We now consider the conditioned version of $\H_p^c$  developed in \cite{JX}.
Let $x = (x_n)_{n \geq 1}$ be a martingale in $L_2(\M) +\M$.
We set (with the convention that $\E_0=\E_1$):
 \[
 s_{c,n} (x) = \Big ( \sum^n_{k = 1} \E_{k-1}|dx_k |^2 \Big )^{1/2}, \quad
 s_c (x) = \Big ( \sum^{\infty}_{k = 1} \E_{k-1}|dx_k |^2 \Big )^{1/2}\,.
 \]
For $2\leq p <\infty$, the column conditioned   martingale Hardy space 
$\h_p^c(\M)$  is defined to be  the space of all martingale $x$ for which $s_c(x)$ belongs to $L_p(\M)$ equipped with the norm  $\| x \|_{\h_p^c}=\| s_c (x) \|_p$. More generally, if $E$ is a symmetric Banach function space with the Fatou property and $E \subseteq L_2 +L_\infty$, we define $\h_E^c (\mathcal{M})$ to  be the set of all martingale $x=(x_n)_{n\geq 1}$ in $E(\M)$ for which  $s_c(x)$ belongs to $E(\M)$).  $\h_E^c (\mathcal{M})$  becomes  a Banach space when equipped with the norm
\[
\| x \|_{\h_E^c}= \| s_c(x)\|_{E(\M)}\,. \]

For  $0< p < 2$ or $E \nsubseteq L_2 +L_\infty$, the definition is more involved.
In this range, we
define $\h_p^c (\mathcal{M})$ to be  the completion of the linear space of finite martingales in $L_p(\M) \cap \M$ under the (quasi) norm $\| x \|_{\h_p^c}=\| s_c (x) \|_p$. We postpone the description of $\h_E^c(\M)$ until after the next discussion.

In the sequel, we will use  more general versions of  these spaces by considering arbitrary sequences in place of martingale difference sequences. For $0<p \leq \infty$,
and a finite sequence $a=(a_n)_{n\geq 1}$ in $L_p(\M)$, we set 
\[
\big\|a\big\|_{L_p(\M;\ell_2^c)}=\Big\| \big(\sum_{n\geq 1} |a_n|^2\big)^{1/2}\Big\|_p.
\] 
The completion (relative to the w*-topology for $p=\infty$) of the space of finite sequences in $L_p(\M)$ equipped with the (quasi) norm $\big\| \cdot \big\|_{L_p(\M;\ell_2^c)}$ will be denoted by $L_p(\M;\ell_2^c)$.   We   will also need  the conditioned $L_p$-spaces which is defined as follows:
for  $0< p \leq \infty$ and  a finite  sequence $a=(a_n)_{n\geq 1}$ in $ L_p(\M)\cap \M$, we set
\[
\big\|a\big\|_{L_p^{\rm cond}(\M;\ell_2^c)} = \Big\| \big( \sum_{n\geq 1} \E_{n-1} (a_n^* a_n) \big)^{1/2}\Big\|_p.
\]
For $0<p<\infty$, the completion of the space of finite  sequences in $ L_p(\M)\cap \M$ equipped with the (quasi) norm $\| \cdot\|_{L_p^{\rm cond}(\M;\ell_2^c)} $ will be denoted by $L_p^{\rm cond}(\M;\ell_2^c)$.  For $p=\infty$, we may define $L_\infty^{\rm cond}(\M;\ell_2^c)$ as the set of all sequences $a=(a_n)_{n\geq1}$  in $\M$ for which the increasing sequence $\big(\sum_{k=1}^n \E_{k-1}(a_k^*a_k) \big)_{n\geq 1}$ is bounded in $\M$. In this case,
\[
\big\|a\big\|_{L_\infty^{\rm cond}(\M;\ell_2^c)} = \sup_{n\geq 1}\Big\| \big( \sum_{k=1}^n \E_{k-1} (a_k^* a_k) \big)^{1/2}\Big\|_\infty.
\]

A very crucial result of Junge \cite{Ju} states that   there  exists an isometric  embedding  of $\h_p^c(\M)$  into  a noncommutative $L_p$-space. Namely, for $0< p \leq \infty$, we have an isometry
\[
U: L_p^{\rm cond}(\M;\ell_2^c) \to L_p(\M \overline{\otimes} B(\ell_2(\mathbb{N}^2)))
\]
with the property that if  $a=(a_n) \in L_p^{\rm cond}(\M;\ell_2^c)$ and $b=(b_n) \in L_q^{\rm cond}(\M;\ell_2^c)$ with $1/p +1/q\le 1$, then 
\[
U(a)^*U(b)=\big( \sum_{n\geq 1} \E_{n-1}(a_n^* b_n) \big) \otimes e_{1,1} \otimes e_{1,1}\,,
\]
where $(e_{i,j})_{i,j\geq 1}$  denotes the unit matrices in $B(\ell_2(\mathbb{N}))$. If we denote by $\cal{D}_c : \h_p^c(\M) \to L_p^{\rm cond}(\M;\ell_2^c)$ the natural map $x \mapsto (dx_n)$, then its composition with $U$ 
 induces  the isometric embedding:
\[
U\cal{D}_c: \h_p^c(\M) \to L_p(\M \overline{\otimes} B(\ell_2(\mathbb{N}^2)))
\]
with the property that  if $x \in \h_p^c(\M)$, $y\in \h_q^c(\M)$, and $1/p +1/q\leq 1$ then 
\[
U\cal{D}_c(x)^* U\cal{D}_c(y)=\big( \sum_{n\geq 1} \E_{n-1}(dx_n^* dy_n) \big) \otimes e_{1,1} \otimes e_{1,1}.
\]
In particular, for $x \in \h_2^c(\M)$,  we have 
\begin{equation}\label{c-square}
|U\cal{D}_c(x)|^2 =(s_c(x))^2 \otimes e_{1,1} \otimes e_{1,1}.
\end{equation}

 Now let $E$ be a symmetric Banach function space with the Fatou property.  We define $\h_E^c(\M)$ to be the set  of all martingales $x=(x_n)_{n\geq 1} \in \h_1^c(\M) + \h_\infty^c(\M)$ for which $U\cal{D}_c(x) \in  E(\M \overline{\otimes} B(\ell_2(\mathbb{N}^2)))$. Then for $x \in \h_E^c(\M)$, we set
\[
\big\|x \big\|_{\h_E^c} :=\big\| U\cal{D}_c(x)\big\|_{E(\M \overline{\otimes} B(\ell_2(\mathbb{N}^2)))}.
\] 
Equipped with  $\|\cdot\|_{\h_E^c}$,  $\h_E^c(\M)$ is a Banach space
 and  
  $U\cal{D}_c$ extends to an isometric embedding of $\h_E^c(\M)$ into $E(\M \overline{\otimes} B(\ell_2(\mathbb{N}^2)))$. We note that if  $E \subseteq  L_2 +L_\infty$, then  the two definitions of $\h_E^c(\M)$ coincide.
  
  \smallskip
  
All definitions and statements above admit corresponding row versions by passing to adjoints. For instance, the row square function of a martingale $x$ is defined as $S_r(x)=S_c(x^*)$, and the row Hardy space $\mathcal{H}_p^r (\mathcal{M})$ consists of all martingales $x$ such that $x^*\in \mathcal{H}_p^c (\mathcal{M})$.
  \smallskip
  
A third type of Hardy spaces that we will use in the sequel are the diagonal Hardy spaces. For $0<p\leq \infty$, we recall that the diagonal Hardy space $\h_p^d(\M)$  is the subspace of $\ell_p(L_p(\M))$ consisting of martingale difference sequences. This definition can be easily extended to the case of symmetric spaces  by setting   $\h_E^d(\M)$ as   the space of all martingales whose martingale difference sequences belong to $E(\M \overline{\otimes} \ell_\infty)$, equipped with the norm $\|x\|_{\h_E^d} := \|(dx_n)\|_{E(\M \overline{\otimes} \ell_\infty)}$.  We will denote by $\cal{D}_d$ the isometric embedding of $\h_E^d$ into $E(\M \overline{\otimes} \ell_\infty)$ given by $x \mapsto (dx_n)_{n\geq 1}$.

We will also make use of  another type of diagonal spaces developed in  \cite{Junge-Perrin}. For $0<p<2$,  a sequence $x=(x_n)$ belongs to $L_p(\M;\ell_1^c)$ if there exist $b_{k,n} \in L_2(\M)$ and $a_{k,n} \in L_q(\M)$ where $1/p=1/2 +1/q$ such that  for every $n\geq 1$,
\begin{equation}\label{factor}
x_n =\sum_{k\geq 1} b_{k,n}^* a_{k,n},
\end{equation}
$\sum_{k,n\geq 1} |b_{k,n}|^2 \in L_1(\M)$, and $\sum_{k,n \geq 1}|a_{k,n}|^2 \in L_{q/2}$. We equip $L_p(\M;\ell_1^c)$ with the (quasi) norm:
\[
\big\| x\big\|_{L_p(\M;\ell_1^c)} =\inf \Big\{ \Big(\sum_{k,n \geq 1} \big\| b_{k,n}\big\|_2^2 \Big)^{1/2} \Big\| \Big( \sum_{k,n \geq 1} |a_{k,n}|^2 \Big)^{1/2} \Big\|_q \Big\},
\]
where the infimum is taken over all factorizations \eqref{factor}. As in \cite[Lemma~6.1.2]{Junge-Perrin}, the unit ball of  $L_p(\M;\ell_1^c)$ coincides with  the set of all sequences $(\b_n \a_n)$ satisfying the following inequality: 
\[
\Big( \sum_{n\geq 1}\|\b_n\|_2^2 \Big)^{1/2} \Big\| \big( \sum_{n\geq 1} |\a_n|^2 \big)^{1/2}\Big\|_q \leq 1.
\]
The following facts are clear from the definitions:  $L_1(\M;\ell_1^c)=\ell_1(L_1(\M))$, $L_p(\M,\ell_1^c) \subseteq \ell_p(L_p(\M))$ for $1<p<2$, and $\ell_p(L_p(\M)) \subseteq L_p(\M;\ell_1^c)$ for $0<p<1$.
The diagonal space  $\h_p^{1_c}(\M)$  is the subspace of 
$L_p(\M;\ell_1^c)$  consisting of  martingale difference sequences.

%%%%%%%%%%%%%%%%%%%%%%%%%%%%%%%%%%%%%%%%%%%%%%%%%%%%%%%%%%%%%%%%%%%%%%%
%%%%%%%%%%%%%%%%%%%%%%%%%%%%%%%%%%%%%%%%%%%%%%%%%%%%%%%%%%%%%%%%%%%%%%%

\section{Davis-type decompositions}

%%%%%%%%%%%%%%%%%%%%%%%%%%%%%%%%%%%%%%%%%%%%%%%%%%%%%%%%%%%%%%%%%%%%%%%
%%%%%%%%%%%%%%%%%%%%%%%%%%%%%%%%%%%%%%%%%%%%%%%%%%%%%%%%%%%%%%%%%%%%%%%

The primary goal  of this section is to provide   extensions of  Davis' decomposition for  adapted sequences in $L_p(\M;\ell_2^c)$ for all $0<p < 2$. Our first result deals with the case  $2/3 \leq p <2$. In this range, we obtain a decomposition with simultaneous control of norms.

%%%%%%%%%%%%%%

\begin{theorem}\label{infinite1} Let $2/3 \leq p < 2$ and   $\xi=(\xi_n)_{n\geq 1}$ be an adapted sequence that belongs to  $L_p(\M;\ell_2^c)\cap L_\infty(\M;\ell_2^c)$. Then there exist   two  adapted sequences $y=(y_n)_{n\geq 1}$ and $z=(z_n)_{n\geq 1}$  such that:
\begin{enumerate}[{\rm(i)}]
\item $\xi=y +z$;
\item $\big\|y\big\|_{\ell_p(L_p(\M))} + \big\| z\big\|_{L_p^{\rm cond}(\M;\ell_2^c)} \leq 2\big(\frac{2}{p}\big)^{1/2}\big\|\xi\big\|_{L_p(\M;\ell_2^c)}$;
\item $\big\|y\big\|_{L_q(\M;\ell_2^c)} + \big\|z\big\|_{L_q(\M;\ell_2^c)} \leq 3 \big\|\xi\big\|_{L_q(\M;\ell_2^c)}$ for every $2\leq q\leq \infty$.
\end{enumerate}
\end{theorem}

For the proof of the theorem, we will need the following lemma  which is an extension of \cite[Lemma~1.1]{PX}.

\begin{lemma}\label{row-lem} Let $2\leq p\leq \infty$ and $1/p=1/q +1/r$. For any sequence $a=(a_n)_{n\geq 1}$ in $L_q(\M)$ and any $A \in L_r(\M)$, we set $B(a,A)=(a_nA)_{n\geq 1}$. Then
\begin{equation}\label{row-Holder}
\Big\| B(a,A)\Big\|_{L_p(\M;\ell_2^r)} \leq \max\Big\{ \big\|a\big\|_{L_q(\M;\ell_2^c) }, \;\big\|a\big\|_{L_q(\M;\ell_2^r)}\Big\} \big\|A\big\|_r.
\end{equation}
\end{lemma}
\begin{proof}
Clearly, $\big\| B(a,A)\big\|_{L_p(\M;\ell_2^r)} =\big\| \sum_{n\geq 1} a_n AA^* a_n^*\big\|_{p/2}^{1/2}$. Denote by $s$   the conjugate index   of $p/2$. By duality, we may fix  $B \in L_s(\M)$ with $B \geq 0$, $\|B\|_s =1$, and such that 
\begin{equation}\label{psi}
\psi(B)=\Big\| \sum_{n\geq 1} a_n AA^* a_n^*\Big\|_{p/2}= \T\Big(\sum_{n\geq 1} a_n AA^* a_n^* B \Big).
\end{equation}
Set $\alpha=(r/2)/(q/2)'$ and $ \beta= s/(q/2)'$ where $(q/2)'$ denotes  the  conjugate index  of $q/2$. One  can readily verify  that $(1-\alpha^{-1})\beta=1$. We will apply the three lines lemma to the analytic function $F$ defined for $0\leq \Re(z) \leq 1$
by 
\[
F(z)=\T\Big(\sum_{n\geq 1} a_n (AA^*)^{\alpha z} a_n^* B^{\beta(1-z)}\Big).
\]
Let $\theta=\alpha^{-1}$ so that $1-\theta=\beta^{-1}$. Then $0\leq \theta \leq 1$ and $F(\theta) =\psi(B)$. Hence, by the three lines lemma, we have 
\begin{equation}\label{3L}
|\psi(B)|=| F(\theta)|\leq \big(\sup_{t \in \mathbb{R}} |F(it)|\big)^{(1-\theta)} \big(\sup_{t \in \mathbb{R}} |F(1+it)|\big)^\theta.
\end{equation}
Using H\"older's inequality, we have 
\begin{equation}\label{t}
\sup_{t \in \mathbb{R}} |F(it)| \leq \sup\Big\{ \Big\| \sum_{n\geq 1} a_n Ua_n^* \Big\|_{q/2} : U \in \M, \|U\| \leq 1\Big\}.
\end{equation}
On the other hand, using the tracial property of $\T$, we also have 
\begin{equation}\label{1t}
\sup_{t \in \mathbb{R}}|F(1+it)|\leq  \sup\Big\{ \Big\| \sum_{n\geq 1} a_n^* Ua_n^* \Big\|_{q/2} : U \in \M, \|U\| \leq 1\Big\}\cdot \big\|A\big\|_r^{2/\theta}.
\end{equation}
As already noted in \cite{PX}, for every operator $U \in \M$ with  $\|U\|\leq 1$,  we have 
 $$\| \sum_{n\geq 1} a_n^* Ua_n\|_{q/2}\leq \| \sum_{n\geq 1} a_n^* a_n\|_{q/2}\;\text{  and }\;\| \sum_{n\geq 1} a_n Ua_n^*\|_{q/2}\leq \| \sum_{n\geq 1} a_n a_n^*\|_{q/2}.$$
Therefore, by combining \eqref{psi} - \eqref{1t},  the desired inequality \eqref{row-Holder} follows.
\end{proof}

\begin{proof}[Proof of Theorem~\ref{infinite1}]
For $n\geq 1$, let $\varsigma_n^2=\sum_{j=1}^n |\xi_j|^2$ and $\varsigma^2=\sum_{j\geq 1} |\xi_j|^2$. Then $(\varsigma_n)$ is an adapted sequence.
Let $\alpha=1-\frac{p}{2}$. For $n\geq 1$,   we let  $w_n=\varsigma_n^{\alpha}$. By approximation,  we assume that the $w_n$'s are invertible. As in the case of martingales, the following  inequality holds:
\begin{equation}\label{L2}
\sum_{n\geq 1} \big\| \xi_n w_n^{-1} \big\|_2^2 \leq \frac{2}{p}\, \big\| \xi\big\|_{L_p(\M;\ell_2^c)}^p.
\end{equation}
This  is implicit in \cite{Bekjan-Chen-Perrin-Y}. Indeed,   
 $$\big\| \xi_n w_n^{-1} \big\|_2^2=\T\big(|\xi_n|^2 w_n^{-2})=\T\big(\varsigma_n^{p-2}(\varsigma_n^2-\varsigma_{n-1}^2)\big)\leq \frac{2}{p}\, \T\big(\varsigma_{n}^p-\varsigma_{n-1}^p\big),$$
where the last inequality comes from \cite{Bekjan-Chen-Perrin-Y}.

Consider now the following  decomposition of $\xi$: for $n\geq 1$, we set
\begin{equation*}
\left\{\begin{split}
y_n &= \xi_n w_n^{-1}(w_n-w_{n-1}),\\
z_n &=\xi_n w_n^{-1}w_{n-1},
\end{split}\right.
\end{equation*}
where we have taken $w_0=0$.
Clearly, $y=(y_n)$ and $z=(z_n)$ are adapted  sequences and for every $n\geq 1$, we have   $\xi_n =y_n + z_n$.
We claim that  $y$ and $z$ satisfy  $\rm(ii)$ and $\rm(iii)$. The argument for $\rm(ii)$ is similar to the proof of \cite[Proposition~6.1.2]{Junge-Perrin}.

 We begin with the diagonal part. We will verify   that $\|y\|_{\ell_p(L_p(\M))} \leq \big(\frac{2}{p}\big)^{1/2} \|\xi\|_{L_p(\M;\ell_2^c)}$. For this, let $1/p=1/2 +1/r$. By H\"older's inequality and \eqref{L2}, we have
\begin{align*}
\big(\sum_{n\geq 1} \big\| y_n\big\|_p^p \big)^{1/p}&\leq  \big(\sum_{n\geq 1} \big\| \xi_nw_n^{-1}(w_n-w_{n-1}) \big\|_p^p\big)^{1/p} \\
&\leq \Big( \sum_{n\geq 1} \big\|\xi_nw_n^{-1}\big\|_2^2\Big)^{1/2} \Big( \sum_{n\geq 1} \big\|w_n-w_{n-1}\big\|_r^r\Big)^{1/r} \\
&\leq \big(\frac{2}{p}\big)^{1/2} \big\|\xi \big\|_{L_p(\M;\ell_2^c)}^{p/2}  \Big\|\big( \sum_{n\geq 1} (w_n-w_{n-1})^r\big)^{1/r} \Big\|_r.
\end{align*}
The crucial fact here is that when  $2/3\leq p <2$, we have $r\geq1$ and therefore  by \cite{Xu-Nikishin}, we get 
\[
 \Big\| \big(\sum_{n\geq 1} (w_n-w_{n-1})^r \big)^{1/r}\Big\|_r\leq  \Big\| \sum_{n\geq 1} w_n-w_{n-1}\Big\|_r =\|\varsigma^{\alpha}\|_r. 
\]
Moreover, as $r=\frac{2p}{2-p}=\frac{p}{\alpha}$, we have $\|\varsigma^{\alpha}\|_r=\|\varsigma\|_p^\alpha= \|\xi\|_{L_p(\M;\ell_2^c)}^\alpha$. Thus, the above estimate becomes
\begin{equation}\label{y1}
\big\| y\big\|_{\ell_p(L_p(\M))} \leq \big(\frac{2}{p}\big)^{1/2} \big\|\xi \big\|_{L_p(\M;\ell_2^c)}.
\end{equation}
 Let us now show that $z \in L_p^{\rm cond}(\M;\ell_2^c)$.  We have,
\begin{align*}
\big\| z \big\|_{L_p^{\rm cond}(\M;\ell_2^c)} &\leq \big\| \big( \sum_{n\geq 1}\E_{n-1}|\xi_nw_n^{-1}w_{n-1}|^2 \big)^{1/2} \big\|_p\\
&=\big\| \big( \sum_{n\geq 1}w_{n-1}\E_{n-1}|\xi_nw_n^{-1}|^2 w_{n-1} \big)^{1/2} \big\|_p.
\end{align*}
Write $w_{n-1}=v_{n-1}\varsigma^{\alpha}$ for some contraction $v_{n-1}$. Then
\[
\big\| z \big\|_{L_p^{\rm cond}(\M;\ell_2^c)} \leq  \Big\| \varsigma^{\alpha}  \Big( \sum_{n\geq 1}   v_{n-1}^* \E_{n-1}|\xi_nw_n^{-1}|^2 v_{n-1} \Big) \varsigma^{\alpha} \Big\|_{p/2}^{1/2}.
\]
Using H\"older's inequality and \eqref{L2}, we deduce that 
\begin{align*}
\big\| z \big\|_{L_p^{\rm cond}(\M;\ell_2^c)}^2 &\leq \big\|\varsigma^{\alpha}\big\|_{p/\alpha}\Big\| \sum_{n\geq 1}   v_{n-1}^* \E_{n-1}|\xi_nw_n^{-1}|^2 v_{n-1} \Big\|_1  \big\|\varsigma^{\alpha}\big\|_{p/\alpha} \\
&\leq \big\| \xi\big\|_{L_p(\M;\ell_2^c)}^{2\alpha} \sum_{n\geq 1} \big\|\xi_nw_n^{-1}\big\|_2^2\\
&\leq  \frac{2}{p} \big\| \xi\big\|_{L_p(\M;\ell_2^c)}^2.
\end{align*}
This shows that $ \big\| z \big\|_{L_p^{\rm cond}(\M;\ell_2^c)}\leq \big(\frac{2}{p}\big)^{1/2} \big\| \xi\big\|_{L_p(\M;\ell_2^c)}$. Combining the last inequality with \eqref{y1}, we obtain $\rm(ii)$.

\medskip

For $\rm(iii)$, we will  only need to verify  that  for $2\leq q\leq \infty$,  $y \in L_q(\M;\ell_2^c)$.
To that end, we observe that  since for $n\geq 1$, $|\xi_n|^2 \leq \varsigma_n^2=w_n^{2/\alpha}$, it is clear   that  $w_n^{2/\alpha}\leq \varsigma^2$.
This fact implies  that  for  every $n\geq 1$,  $w_n^{-1} |\xi_n|^2 w_n^{-1} \leq w_n^{(2/\alpha)-2}=(w_n^{2/\alpha})^{1-\alpha} \leq \varsigma^{2-2\alpha}=\varsigma^{p}$.
Therefore,  the following estimate clearly follows: 
\begin{align*}
\Big\| \big(\sum_{n\geq 1}| \xi_nw_n^{-1}(w_n-w_{n-1})|^2 \big)^{1/2}\Big\|_q &\leq 
\Big\|  \big( \sum_{n\geq 1}(w_n-w_{n-1}) \varsigma^{p} (w_n-w_{n-1})\big)^{1/2}\Big\|_q\\
&= \Big\| \big\{ (w_n-w_{n-1})\varsigma^{p/2}\big\}_{n\geq 1}\Big\|_{L_q(\M; \ell_2^r)}.
\end{align*}
Using  Lemma~\ref{row-lem} and \cite{Xu-Nikishin}, we deduce that
\begin{align*}
\big\|y\big\|_{L_q(\M ; \ell_2^c)} &\leq \big\|\varsigma^{p/2}\big\|_{q/(1-\alpha)}\,\big\|(w_n-w_{n-1})_{n\geq 1}\big\|_{L_{q/\alpha}(\M;\ell_2^c)}\\
&\leq \big\|\varsigma^{p/2}\big\|_{q/(1-\alpha)}\,\big\| \sum_{n\geq 1}(w_n-w_{n-1}) \big\|_{q/\alpha}\\
&\leq \big\|\xi\big\|_{L_q(\M;\ell_2^c)}^{1-\alpha} \big\| \varsigma^\alpha\big\|_{q/\alpha}\\
&\leq \big\|\xi\big\|_{L_q(\M; \ell_2^c)}\,.
\end{align*}
This proves that 
 $\big\| y \big\|_{L_q(\M;\ell_2^c)}  \leq \big\| \xi \big\|_{L_q(\M;\ell_2^c)}$.
  The corresponding inequality for $z$ easily follows from triangle inequality. Thus the proof is complete.
\end{proof}

We remark that  for $1<p<2$,  the norm used for the  diagonal part  $\| \cdot\|_{\ell_p(L_p(\M))}$  in item $\rm(ii)$ can be improved to $\| \cdot\|_{L_p(\M; \ell_1^c)}$. This was already known to \cite{Junge-Perrin}.

We now consider  the remaining case $0<p<2/3$. At the time of this writing, we are unable to provide a simultaneous control of norms in the spirit of Theorem~\ref{infinite1}.  We also do not know if the  quasi norm  $\big\|\cdot \big\|_{L_p(\M;\ell_1^c)}$ used below can be improved to  $\big\|\cdot\big\|_{\ell_p(L_p(\M))}$.

\begin{theorem}\label{infinite2} Let  $0<p< 2/3$ and $\xi=(\xi_n)_{n\geq 1}$ be an adapted sequence that belongs to  $L_p(\M;\ell_2^c)$. Then there exist   two  adapted sequences $y=(y_n)_{n\geq 1}$ and $z=(z_n)_{n\geq 1}$  such that:
\begin{enumerate}[{\rm(i)}]
\item $\xi=y +z$;
\item $\big\|y\big\|_{L_p(\M;\ell_1^c)} + \big\| z\big\|_{L_p^{\rm cond}(\M;\ell_2^c)} \leq 2\big(\frac{2}{p}\big)^{1/2}\big\|\xi\big\|_{L_p(\M;\ell_2^c)}$.
\end{enumerate}
\end{theorem}
\begin{proof}
We follow the notation  from the proof of Theorem~\ref{infinite1}.
This time  the decomposition is given as follows:
\begin{equation}\label{decomp2}
\left\{\begin{split}
y_n &= \xi_n w_n^{-2}(w_n^2-w_{n-1}^2),\\
z_n &=\xi_n w_n^{-2}w_{n-1}^2.
\end{split}\right.
\end{equation}
Note that   compared to the previous one, we use $w_n^2$ in place of $w_n$. This adjustment is mainly needed  for the diagonal part. As before, we clearly have that  $y=(y_n)$ and $z=(z_n)$ are adapted  sequences with  $\xi_n =y_n + z_n$ for $n\geq 1$.
We claim that  $y$ and $z$ satisfy  $\rm(ii)$. The argument is an adaptation of that used in the proof of  item $\rm(ii)$  from Theorem~\ref{infinite1}, so we only  present a sketch.

We verify first that $y \in L_p(\M;\ell_1^c)$.  Define $r$ by  $1/p=1/2+1/r$. For $n\geq 1$, write  $y_n=\beta_n \alpha_n$ where $\beta_n=\xi_n w_n^{-2}(w_n^2-w_{n-1}^2)^{1/2}$ and $\alpha_n=(w_n^2 -w_{n-1}^2)^{1/2}$. It follows that
\begin{align*}
\big\| y\big\|_{L_p(\M;\ell_1^c)} &=   \Big\|(\beta_n\alpha_n)_{n\geq 1} \big)\Big\|_{L_p(\M;\ell_1^c)} \\
&\leq \Big( \sum_{n\geq 1} \big\|\xi_nw_n^{-2}(w_n^2-w_{n-1}^2)^{1/2}\big\|_2^2\Big)^{1/2} \, \Big\| \big(\sum_{n\geq 1} |(w_n^2-w_{n-1}^2)^{1/2}|^2 \big)^{1/2}\Big\|_r \\
&= \Big( \sum_{n\geq 1} \big\|\xi_nw_n^{-1}[w_n^{-1}(w_n^2-w_{n-1}^2)^{1/2}]\big\|_2^2\Big)^{1/2}  \,\Big\| \big(\sum_{n\geq 1} (w_n^2-w_{n-1}^2) \big)^{1/2}\Big\|_r.
\end{align*}
Observing that  $\{w_n^{-1}(w_n^2 -w_{n-1}^2)^{1/2}\}_{n\geq 1}$ is a sequence of contractions, we have 
\[
\big\| y\big\|_{L_p(\M;\ell_1^c)}\leq \Big( \sum_{n\geq 1} \big\|\xi_nw_n^{-1}\big\|_2^2\Big)^{1/2} \Big\| \big(\sum_{n\geq 1} (w_n^2-w_{n-1}^2) \big)^{1/2}\Big\|_r.
\]
Moreover, as $r=\frac{2p}{2-p}=\frac{p}{\alpha}$, we have
\[
 \Big\| \big(\sum_{n\geq 1} w_n^2-w_{n-1}^2 \big)^{1/2}\Big\|_r= \|\varsigma^{\alpha}\|_r=\|\varsigma\|_p^\alpha= \|\xi\|_{L_p(\M;\ell_2^c)}^\alpha.
\]
Using \eqref{L2}, the  above estimate leads to:
\begin{equation*}
\big\| y\big\|_{L_p(\M;\ell_1^c)} \leq \big(\frac{2}{p}\big)^{1/2} \big\|\xi \big\|_{L_p(\M;\ell_2^c)}.
\end{equation*}
On the other hand,    we  may also  write:
\begin{align*}
\big\| z \big\|_{L_p^{\rm cond}(\M;\ell_2^c)} &\leq \big\| \big( \sum_{n\geq 1}\E_{n-1}|\xi_nw_n^{-2}w_{n-1}^2|^2 \big)^{1/2} \big\|_p\\
&=\big\| \big( \sum_{n\geq 1}w_{n-1}\E_{n-1}|\xi_nw_n^{-2} w_{n-1}|^2 w_{n-1} \big)^{1/2} \big\|_p.
\end{align*}
As before, we  may deduce that 
\begin{equation*}
\big\| z \big\|_{L_p^{\rm cond}(\M;\ell_2^c)}^2 
\leq  \big\| \xi\big\|_{L_p(\M;\ell_2^c)}^{2\alpha} \sum_{n\geq 1} \big\|\xi_nw_n^{-2} w_{n-1}\big\|_2^2.
\end{equation*}
Using \eqref{L2} and the fact that  $(w_n^{-1}w_{n-1})_{n\geq 1}$ is a sequence of contractions,  we conclude that 
\[
\big\| z \big\|_{L_p^{\rm cond}(\M;\ell_2^c)}^2 \leq  \big\| \xi\big\|_{L_p(\M;\ell_2^c)}^{2\alpha}  \sum_{n\geq 1} \big\|\xi_nw_n^{-1}\big\|_2^2 
\leq  \frac{2}{p} \,\big\| \xi\big\|_{L_p(\M;\ell_2^c)}^2
\]
 The proof is complete.
\end{proof}

\bigskip

The following noncommutative Davis decomposition easily follows from
Theorem~\ref{infinite1} when $1\leq p<2$. A notable new feature of this decomposition is a simultaneous control on the $\H_p^c$ and  $\H_q^c$ norms for $2\leq q<\infty$.  Decompositions of such a nature are important for some aspects of noncommutative martingale theory, for instance, for the analytic theory of quantum stochastic integrals and  the study of noncommutative martingales in symmetric spaces.

\begin{corollary}\label{davis:1} Let  $1\leq p<2\leq q<\infty$.
Then every  martingale $x \in \H_p^c(\M) \cap \H_q^c(\M)$ admits a  decomposition into  two martingales $x^c$ and $x^d$ such that:
\begin{enumerate}[{\rm (i)}]
 \item   $x=x^c +x^d$;
  \item $\big\|x^c\big\|_{\h_p^c} +  \big\|x^d\big\|_{\h_p^d}
  \leq 2^{5/2}\big\|x\big\|_{\H^c_p}$;
  \item $\big\|x^c\big\|_{\h_q^c} +  \big\|x^d\big\|_{\h_q^d}
  \leq Cq \big\|x\big\|_{\H^c_q}$ with an absolute constant $C$.
\end{enumerate}
\end{corollary}

\begin{proof} Let $\xi=(dx_n)_{n\geq 1}$.
 It is enough to take
martingales $x^c$ and $x^d$ with $dx_n^c= z_n - \E_{n-1}(z_n)$ and $dx_n^d= y_n - \E_{n-1}(y_n)$, where $y$ and $z$ are the adapted sequences from Theorem~\ref{infinite1}. Clearly, $x=x^d +x^c$. From the facts that $\E_{n-1}$'s are contractive projections  on $L_p(\M)$, we get that $\|x^d\|_{\h_p^d} \leq 2\|y\|_{\ell_p(L_p(\M))}$. Also, since for every $n\geq 1$, $\E_{n-1}|dx_n^c|^2 \leq \E_{n-1}|z_n|^2$, we immediately get  $\|x^c\|_{\h_p^c} \leq \|z\|_{L_p^{\rm cond}(\M;\ell_2^c)}$. For the $\h_q^c$-norm of $x^c$, we have from the noncommutative Stein inequality  and Theorem~\ref{infinite1}(iii) that
 \[
\big\|x^c\big\|_{\h_q^c} \leq  \big\| z \big\|_{L_q^{\rm cond}(\M;\ell_2^c)}
\leq \gamma_q \big\| z \big\|_{L_q(\M;\ell_2^c)}
\leq 3\gamma_q \big\| x\big\|_{\H_q^c}.
\]
The order of $\gamma_q$ is from \cite{JX2}. The estimate on $\|x^d\|_{h_q^d}$ follows from the fact that for $q\ge2$, the identity map from $L_q(\M; \ell_2^c)$ into $\ell_q(L_q(\M))$ is a contraction.
\end{proof}

\begin{remark}\label{davis:simul}
Corollary~\ref{davis:1} goes beyond the version of the  noncommutative Davis' decomposition of  Perrin (\cite[Theorem~2.1]{Perrin}) since for any given $2\leq q_0<\infty$,  it provides through the use of interpolation a  decomposition  that works simultaneously for all $1\le p \leq q_0$ with universal  constants.  This fact is essential for the applications in the next section.
 \end{remark}

 An application of the proof of Theorem~\ref{infinite2} to the case $p=1$ provides  uniform previsible estimates. The next result  should be compared with \cite[Theorem~III.3.5]{Ga}
 and \cite[Theorem~4.4]{Muller2}.
 \begin{corollary}\label{davis:2}
 Every martingale $x \in \H_1^c(\M) $ admits a  decomposition into  two martingales $x^c$ and $x^d$ such that:
\begin{enumerate}[{\rm (i)}]
 \item   $x=x^c +x^d$;
  \item $\big\|x^c\big\|_{\h_1^c} +  \big\|x^d\big\|_{\h_1^d}
  \leq 2^{5/2}\big\|x\big\|_{\H^c_1}$;
  \item the martingale $x^c$ satisfies the previsible uniform estimates:
  \[
  |dx_n^c|^2 \leq 2 S_{c,n-1}^2(x),\quad n\geq 1.
  \]
\end{enumerate}
\end{corollary} 
 \begin{proof}
 It suffices to apply the decomposition \eqref{decomp2} to $p=1$. That is, we set for $n\geq 1$,
 \begin{equation*}
 \left\{\begin{split}
dx_n^d &= dx_nS_{c,n}^{-1}(S_{c,n}-S_{c,n-1}) - \E_{n-1}[dx_nS_{c,n}^{-1}(S_{c,n}-S_{c,n-1})]\\
dx_n^c  &=dx_nS_{c,n}^{-1}S_{c,n-1} - \E_{n-1}[dx_nS_{c,n}^{-1}S_{c,n-1}].
\end{split}\right.
\end{equation*}
The verification of (ii) follows exactly the proof of Theorem~\ref{infinite2}. For the previsible estimates, we have for $n\geq 1$,
 $$|dx_n^c|^2 \leq 2(|dx_nS_{c,n}^{-1}S_{c,n-1}|^2 +|\E_{n-1}[dx_nS_{c,n}^{-1}S_{c,n-1}]|^2).$$ 
Since 
 $$|dx_nS_{c,n}^{-1}S_{c,n-1}|^2=S_{c,n-1} S_n^{-1}(S_{c,n}^2-S_{c,n-1}^2)S_{c,n}^{-1}S_{c,n-1} \leq S_{c,n-1}^2\,,$$
we clearly get the desired estimate.
 \end{proof}

\begin{remark}
Corollaries~\ref{davis:1} and \ref{davis:2}  are not valid for the case $0<p<1$ even for classical martingales. Indeed, if Corollary~\ref{davis:1} was valid for $0<p<1$, we would then have $\h_p^{1_c}(\M) +\h_p^c(\M) = \H_p^c(\M)$. On the other  hand, one can easily see that  $\h_p^{1_c}(\M) \subseteq L_p(\M)$ and  for classical martingales, we also have $\h_p^c(\M) \subseteq L_p(\M)$. Thus, there would exist a constant $C_p$ such that  $\|x\|_p \leq C_p \|x\|_{\H_p}$ for all (classical) martingales $x \in \H_p$. However, 
such constant $C_p$ does not exist (see \cite[Example~8.1]{Burk-Gundy}). 
\end{remark}

Our next result shows that  our decomposition  in Theorem~\ref{infinite1} for $p=1$ is stronger than the noncommutative L\'epingle-Yor inequality.

\begin{corollary}\label{Lepingle}
Let  $(\xi_n)_{n\geq 1} $ be an adapted sequence in $L_1(\M)$. Then we have
\[
\Big\| \big( \sum_{n\geq 1} | \E_{n-1}(\xi_n) |^2\big)^{1/2} \Big\|_1 \leq  2\sqrt{2} \,
\Big\| \big( \sum_{n\geq 1} | \xi_n |^2\big)^{1/2} \Big\|_1.
\]
\end{corollary}

For classical martingales, the inequality above is known as the L\'epingle-Yor inequality (\cite{Lepingle}).
Its noncommutative analogue  as stated in Corollary~\ref{Lepingle}  was  proved in \cite{Qiu1} with constant $2$. Unfortunately,  our  alternative proof below yields  only the constant  $2\sqrt{2}$. We should note that the optimal constant for the classical situation is $\sqrt{2}$ (see \cite{Osekowski1}).

\begin{proof}[Proof of Corollary~\ref{Lepingle}]
Let $\xi=(\xi_n)$ be an adapted sequence in $L_1(\M;\ell_2^c)$. Apply Theorem~\ref{infinite1} to get a decomposition  into two adapted sequences $y$ and $z$ such that: 
\[
\big\| y\big\|_{\ell_1(L_1(\M))} + \big\|z \big\|_{L_1^{\rm cond}(\M;\ell_2^c)} \leq 2\sqrt{2}\, \big\| \xi\big\|_{L_1(\M;\ell_2^c)}.
\]
Note that since we are not using item~(iii) of Theorem~\ref{infinite1}, the assumption that  the adapted sequence $\xi$ belongs to $L_\infty(\M;\ell_2^c)$ is not needed. Since 
$|\E_{n-1}(b)|^2 \leq \E_{n-1}(|b|^2)$ for every $b \in L_2(\M)$, it follows that 
\begin{equation}\label{z}
\Big\| \big( \sum_{n\geq 1} | \E_{n-1}(z_n) |^2\big)^{1/2} \Big\|_1 \leq \big\|z\big\|_{L_1^{\rm cond}(\M;\ell_2^c)}.
\end{equation}
Moreover, as  $\ell_1(L_1(\M))$  embeds contractively into $L_1(\M;\ell_2^c)$ and the expectations $\E_n$'s are contraction in $L_1(\M)$, we also have
\begin{equation}\label{y}
\Big\| \big( \sum_{n\geq 1} | \E_{n-1}(y_n) |^2\big)^{1/2} \Big\|_1 \leq \sum_{n\geq 1} \big\| \E_{n-1}(y_n) \big\|_1 \leq \sum_{n\geq 1} \big\|y_n\big\|_1.
\end{equation}
Combining \eqref{z} and \eqref{y}, we deduce that
\[
\Big\| \big( \sum_{n\geq 1} | \E_{n-1}(\xi_n) |^2\big)^{1/2} \Big\|_1 \leq \big\| y\big\|_{\ell_1(L_1(\M))} + \big\|z \big\|_{L_1^{\rm cond}(\M)}. 
\]
This proves the desired inequality. 
\end{proof}

In the following, we  extend the Davis decomposition to the case of martingales in a certain class of noncommutative symmetric spaces. This is one of the main tools that we use in the next section.

\begin{theorem}\label{E-davis} Let $1<p<q<\infty$ and 
 $E$ be a symmetric Banach function space with the Fatou property and  $E\in {\rm Int}(L_p, L_q)$.  
 There exist two positive constants  $\a_E$ and $\b_E$ such that:
 \begin{enumerate}[{\rm(i)}]
\item  
 for every $x \in \H_E^c(\M)$,  the following inequality holds:
 \begin{equation*}
\a_E^{-1} \inf\Big\{ \big\|x^d \big\|_{\h_E^d} +\big\|x^c \big\|_{\h_E^c}\Big\} \leq 
\big\|x\big\|_{\H_E^c}, 
\end{equation*}
where the infimum   is taken over all $x^d \in \h_E^d(\M)$ and  $x^c \in \h_E^c(\M)$ such that $x=x^d +x^c$.
\item 
for every $x \in \h_E^d \cap \h_E^c(\M)$, the following inequality holds:
 \begin{equation*}
\big\|x\big\|_{\H_E^c} \leq \b_E \max\Big\{ \big\| x\big\|_{\h_E^d}, \big\|x \big\|_{\h_E^c}\Big\}.
\end{equation*}
\end{enumerate}
\end{theorem}

\begin{proof}
Throughout the proof we make use of the following notations for  the compatible  couples
\[
\overline{A} :=(L_p(\M\overline{\otimes} \ell_\infty), L_q(\M\overline{\otimes} \ell_\infty))\  \text{ and }\    \overline{B} :=(L_p(\M \overline{\otimes} B(\ell_2(\mathbb{N}^2))) , L_q(\M \overline{\otimes} B(\ell_2(\mathbb{N}^2)))).
\] 
Also we may fix a symmetric function space $F$  with nontrivial Boyd indices such that for any semifinite von Neumann algebra $\N$, we have $E(\N)=[ L_p(\N), L_q(\N)]_{F, \underline{j}}$ as stated in Lemma~\ref{interpolation1}.  We note that when  $1<p<q<\infty$, $\H_p^c(\M)$ and $\H_q^c(\M)$  embed complementedly  into $L_p(\M \overline{\otimes} B(\ell_2))$ and $L_q(\M \overline{\otimes} B(\ell_2))$, respectively. 
This implies that for every $x \in \H_p^c(\M) +\H_q^c(\M)$ and $t>0$,
\[
K\big(t,x; \H_p^c(\M), \H_q^c(\M)\big) \leq C_{p,q} K\big(t,\cal{D}_c(x); L_p(\M \overline{\otimes} B(\ell_2)),L_q(\M \overline{\otimes} B(\ell_2))\big).
\]
As a consequence, we have   that for every $x \in \H_E^c(\M)$, 
\[
\big\| x\big\|_{\H_E^c} = \big\|\cal{D}_c(x) \big\|_{E(\M \overline{\otimes} B(\ell_2))} \approx_E  \big\| x\big\|_{F,\underline{j}},
\]
where the interpolation on the last norm is  taken with respect to the couple $(\H_p^c(\M), \H_q^c(\M))$.
%%%

We are now ready to present the proof of (i).
 Let $x \in \H_E^c(\M)$ and fix a representation  $x =\sum_{\nu \in \mathbb{Z}} u_\nu$  in the compatible couple $(\H_p^c(\M), \H_q^c(\M))$ such that:
\begin{equation}\label{rep-norm}
 \Big\| \underline{j}\big(\cdot,\{u_\nu\}_\nu\big)\Big\|_F \leq 2 \big\|x \big\|_{F, \underline{j}}.
 \end{equation}
We recall that for every $\nu \in \mathbb{Z}$, $u_\nu \in \H_p^c(\M) \cap \H_q^c(\M)$. By Corollary~\ref{davis:1}, there exist $a_\nu \in \h_p^d(\M)\cap \h_q^d(\M)$ and  $b_\nu \in \h_p^c(\M) \cap  \h_q^c(\M)$,  satisfying:
\begin{equation}
u_\nu=a_\nu +b_\nu
\end{equation}
and if $s\in \{p,q\}$, then
\begin{equation}\label{double-estimate}
\big\|a_\nu \big\|_{\h_s^d} + \big\|b_\nu \big\|_{\h_s^c}  \leq 
C(p,q)  \big\| u_\nu \big\|_{\H_s^c}.
\end{equation}
For each $\nu \in \mathbb{Z}$,  we consider $\mathcal{D}_d(a_\nu)\in \Delta(\overline{A})$ and  
$U\mathcal{D}_c(b_\nu) \in \Delta(\overline{B})$.
As in \cite{RW}, inequality  \eqref{double-estimate} can be reinterpreted by using the $J$-functionals as follows:
 \begin{equation}\label{J-a1}
\left\{\begin{split}
J\big( t,\mathcal{D}_d(a_\nu); \overline{A}\big) &\leq C(p,q) J(t,u_\nu),\  
t>0,\\
 J\big(t,U\mathcal{D}_c(b_\nu); \overline{B}\big) &\leq C(p,q) J(t,u_\nu), \ t>0.
 \end{split}\right.
\end{equation}
One can show as in \cite[Sublemma~3.3]{RW} that the series $\sum_{\nu \in \mathbb{Z}} \cal{D}_d(a_\nu)$ is weakly unconditionally Cauchy in $E(\M\overline{\otimes}\ell_\infty)$. A fortiori, it is convergent in $\Sigma(\overline{A})$. Similarly, we can also get that $\sum_{\nu \in \mathbb{Z}} U\cal{D}_c(b_\nu)$ is convergent in $\Sigma(\overline{B})$. Set
\[
  \displaystyle{\alpha := \sum_{\nu \in \mathbb{Z}} \mathcal{D}_d(a_\nu) \in\Sigma(\overline{A})}\   \text{ and }\ 
\displaystyle{\beta := \sum_{\nu \in \mathbb{Z}} U\mathcal{D}_c(b_\nu) \in 
\Sigma(\overline{B})}.
\] 
The series $\sum_{\nu \in \mathbb{Z}} \cal{D}_d(a_\nu)$ may be viewed as a representation of $\alpha$ in the interpolation couple $\overline{A}$. Similarly, 
$\sum_{\nu \in \mathbb{Z}} U\cal{D}_c(b_\nu)$ is a representation of $\beta$  in the interpolation couple $\overline{B}$. We claim that:
\begin{equation}\label{d-inequality}
 \big\|  \alpha \big\|_{E(\M \overline{\otimes} \ell_\infty)} +
\big\| \beta\big\|_{E(\M \overline{\otimes} B(\ell_2(\mathbb{N}^2)))} \leq \kappa_E\big\|x\big\|_{\H_E^c}.
\end{equation} 
To see this claim, we observe from \eqref{J-a1}  that for every $s>0$,
\[
 \underline{j}(s,\{\cal{D}_d(a_\nu)\}_\nu;\overline{A}) + \underline{j}(s,\{U\cal{D}_c(b_\nu)\}_\nu; \overline{B})  \leq 2C(p,q) \underline{j}(s,\{u_\nu\}_\nu).
\]
Taking the norms on  the function space $F$ together with  \eqref{rep-norm} gives
\[
\big\|  \alpha \big\|_{F, \underline{j}} +
\big\| \beta\big\|_{F, \underline{j}} \leq 4C(p,q)\big\| x\big\|_{F, \underline{j}}.
\]
This proves \eqref{d-inequality}. To conclude the proof, it is plain that 
there exist $a\in \h_E^d(\M)$ and $b\in \h_E^c(\M)$ such that 
$\a=\mathcal{D}_d(a)$ and  $\b=U\mathcal{D}_c(b)$.  Moreover, it is clear from the construction that $x=a+b$. Indeed, the fact that $\alpha$ is a martingale difference sequence follows from the convergence of the series $\sum_{\nu \in \mathbb{Z}} \cal{D}_d(a_\nu)$ in $\Sigma(\overline{A})$. Similarly,  the representation above also gives that $\beta$ is in the range of $U\cal{D}_c$ in $\Sigma(\overline{B})$. We may now conclude  from \eqref{d-inequality} that 
\begin{equation*}
 \big\|  a \big\|_{\h_E^d(\M)} +
\big\| b\big\|_{\h_E^c(\M)} \leq \kappa_E\big\|x\big\|_{\H_E^c}.
\end{equation*} 
The proof of (i) is complete.
Item~(ii) can be obtained by duality in the same manner as in Part~III of the proof of \cite[Theorem~3.1]{RW}. The details are left to the reader.
\end{proof}

As an immediate application  of Theorem~\ref{E-davis}, we have the following result:

\begin{corollary}\label{appl1} Let $E$ be a symmetric Banach function space with the Fatou property.
\begin{enumerate}[{\rm(i)}]
\item If $E  \in {\rm Int}(L_p, L_2)$  for some $1<p<2$, then 
$\displaystyle{
\H_E^c(\M)=\h_E^d(\M) + \h_E^c(\M)}$
with equivalent norms.
\item If $E \in {\rm Int}(L_2, L_q)$ for some $2<q<\infty$, then 
$\displaystyle{
\H_E^c(\M)=\h_E^d(\M) \cap \h_E^c(\M)}$
with equivalent norms.
\end{enumerate}
\end{corollary}

\begin{proof} 
We already have from Theorem~\ref{E-davis} that $\H_E^c(\M) \subseteq \h_E^d(\M) + \h_E^c(\M)$. The reverse inclusion  follows from the fact that   for $1<v\leq 2$, $\h_v^d(\M) +\h_v^c(\M)=  \H_v^c(\M)$. Indeed, this identification implies in  particular that $\h_v^d(\M) \subset \H_v^c(\M)$ and $\h_v^c(\M)\subseteq \H_v^c(\M)$. A standard use of interpolation then gives $\h_E^d(\M) \subseteq \H_E^c(\M)$ and $\h_E^c(\M)\subseteq \H_E^c(\M)$. We should recall here that since  $E\in  {\rm Int}(L_p, L_2)$  and $1<p<2$, by complementation, we have 
$\h_E^d(\M)\in  {\rm Int}(\h_p^d(\M), \h_2^d(\M))$, $\h_E^c(\M)\in  {\rm Int}(\h_p^c(\M), \h_2^c(\M))$, and $\H_E^c(\M)\in  {\rm Int}(\H_p^c(\M), \H_2^c(\M))$. 
The argument for (ii)  is identical.
\end{proof}

\begin{remark} We do not know  if the statement in the first item  of Theorem~\ref{E-davis} (respectively of Corollary~\ref{appl1}) remains  valid if one only assumes that $E\in {\rm Int}(L_1,L_q)$ for $1<q<\infty$ (respectively $E\in {\rm Int}(L_1,L_2)$ ). Since both assertions hold for $E=L_1$, it is reasonable to conjecture that this should be the case in general. We leave these as  open problems.
\end{remark}

\begin{remark} We conclude this section by calling to the reader's attention    that all statements above admit  corresponding row versions.
\end{remark}

%%%%%%%%%%%%%%%%%%%%%%%%%%%%%%%%%%%%%%%%%%%%%%%%%%%%%%%%%%%%%%%%%%%%%%%
%%%%%%%%%%%%%%%%%%%%%%%%%%%%%%%%%%%%%%%%%%%%%%%%%%%%%%%%%%%%%%%%%%%%%%%

\section{Applications to noncommutative Burkholder/Rosenthal inequalities}

%%%%%%%%%%%%%%%%%%%%%%%%%%%%%%%%%%%%%%%%%%%%%%%%%%%%%%%%%%%%%%%%%%%%%%%
%%%%%%%%%%%%%%%%%%%%%%%%%%%%%%%%%%%%%%%%%%%%%%%%%%%%%%%%%%%%%%%%%%%%%%%

%%%%%%%%%%%%%%%%%%%%%%%%%%%%%%%%%%%%%%%%%%%%%%%%%%%%%%%%%%%%%%%%%%%%%%%
\subsection{The case of noncommutative symmetric spaces}
%%%%%%%%%%%%%%%%%%%%%%%%%%%%%%%%%%%%%%%%%%%%%%%%%%%%%%%%%%%%%%%%%%%%%%%

The main objective of this subsection  is to provide    versions of the Burkholder/Rosenthal inequality for martingales in noncommutative symmetric spaces. Recall that the noncommutative Burkholder/Rosenthal inequalities  were proved by Junge and the third named author for martingales in noncommutative $L_p$-spaces for $1<p<\infty$ (\cite{JX}). We should emphasize here that it is essential to separate the case $1<p \leq 2$ from the one $2\leq p <\infty$. In fact, the original classical case was only proved for $2\leq p<\infty$ and it was  in \cite{JX} that the corresponding case $1<p<2$ was discovered. Our main result in this subsection strengthens the  versions of Burkholder/Rosenthal  inequalities  for noncommutative symmetric spaces from \cite{Dirksen2,RW}. It reads as follows:

\begin{theorem}\label{appl-Burk} Let $E$ be a symmetric Banach function space with the Fatou property.
\begin{enumerate}[{\rm(i)}]
\item If $E  \in {\rm Int}(L_p, L_2)$ for some $1<p<2$,  then 
$\displaystyle{
E(\M)=\h_E^d(\M) + \h_E^c(\M) +\h_E^r(\M)}$
with equivalent norms.
\item If $E \in {\rm Int}(L_2, L_q)$ for some $2<q<\infty$, then 
$\displaystyle{
E(\M)=\h_E^d(\M) \cap \h_E^c(\M)\cap \h_E^r(\M)}$
with equivalent norms.
\end{enumerate}
\end{theorem}
We remark that (ii) was also obtained  recently by   Jiao  {\it  et al.} in 
\cite[Theorem~1.5]{Jiao-Sukochev-Zanin-Zhou} under the more restrictive assumption that $E \in {\rm Int}(L_2, L_4)$. It is important to note that through the use  of interpolation, it is not difficult to deduce  that if $1<p<2$ and $E  \in {\rm Int}(L_p, L_2)$ then 
$ \h_E^w(\M)\subseteq E(\M)$ for $w\in\{d,c,r\}$. Therefore, we always have in this case that $\h_E^d(\M) +\h_E^c(\M) + \h_E^r(\M) \subseteq E(\M)$. Similarly, when $2 < q<\infty$
and $E\in {\rm Int}(L_2, L_q)$  then $E(\M) \subseteq \h_E^d(\M) \cap \h_E^c(\M) \cap \h_E^r(\M)$. Thus, it is only necessary to prove the respective  reverse inclusions.
To this end, we will prove the following more general result:
\begin{theorem}\label{appl-Burk2} Let $1<p<q<\infty$ and 
 $E$ be a symmetric Banach function space with the Fatou property and   such that 
 $E\in {\rm Int}(L_p, L_q)$.  
Then there exist positive constants  $\delta_E$  and $\eta_E$ such that:
 \begin{enumerate}[{\rm(i)}]
\item 
 for every $x \in E(\M)$,  the following inequality holds:
 \begin{equation*}
\delta_E^{-1} \inf\Big\{ \big\|x^d \big\|_{\h_E^d} +\big\|x^c \big\|_{\h_E^c} + \big\|x^r \big\|_{\h_E^r}\Big\} \leq 
\big\|x\big\|_{E(\M)},
\end{equation*}
where the infimum   is taken over all $x^d \in \h_E^d(\M)$,   $x^c \in \h_E^c(\M)$, and $x^r \in \h_E^r(\M)$  such that $x=x^d +x^c +x^r$;
\item for every $x \in \h_E^d(\M)  \cap \h_E^c(\M) \cap \h_E^r(\M)$, the following inequality holds:
 \begin{equation*}
\big\|x\big\|_{E(\M)} \leq \eta_E \max\Big\{ \big\| x\big\|_{\h_E^d}, \big\|x \big\|_{\h_E^c}, \big\|x \big\|_{\h_E^r}\Big\}.
\end{equation*}
\end{enumerate} 
\end{theorem}
Clearly, Theorem~\ref{appl-Burk2} and the preceding discussion imply Theorem~\ref{appl-Burk}. Our strategy for the proof of Theorem~\ref{appl-Burk2} is to use the corresponding Burkholder-Gundy along side our Davis decomposition stated in Theorem~\ref{E-davis}.  The following version of the Burkholder-Gundy inequalities  is implicit in \cite{Dirk-Pag-Pot-Suk}:
\begin{proposition}\label{BG} Let $1<p<q<\infty$ and 
 $E$ be a symmetric Banach function space with the Fatou property such that 
 $E\in {\rm Int}(L_p, L_q)$.  
Then there exist positive constants  $C_E$  and $c_E$ such that:
 \begin{enumerate}[{\rm(i)}]
\item 
 for every $x \in E(\M)$,  the following inequality holds:
 \begin{equation*}
 \inf\Big\{ \big\|x^c \big\|_{\H_E^c} + \big\|x^r \big\|_{\H_E^r}\Big\} \leq 
C_E\big\|x\big\|_{E(\M)}, 
\end{equation*}
where the infimum   is taken over all   $x^c \in \H_E^c(\M)$, and $x^r \in \H_E^r(\M)$  such that $x=x^c +x^r$;
\item for every $x \in \H_E^c(\M) \cap \H_E^r(\M)$, the following inequality holds:
 \begin{equation*}
\big\|x\big\|_{E(\M)} \leq c_E \max\Big\{  \big\|x \big\|_{\H_E^c}, \big\|x \big\|_{\H_E^c}\Big\}.
\end{equation*}
\end{enumerate} 
\end{proposition}
\begin{proof} Assume that   $E\in {\rm Int}(L_p,L_q)$  with  $1<p<q<\infty$. By the boundedness of martingale transforms on $E(\M)$ (\cite[Proposition~4.9]{Ran15}), there exists a constant $\kappa_E$ such that  for any given finite  martingale $x$  in $E(\M)$,
\[
\kappa_E^{-1} \mathbb{E}\big\| \sum_{n\geq 1} \epsilon_n dx_n \big\|_{E(\M)}
 \leq \big\| x\big\|_{E(\M)} \leq \kappa_E \mathbb{E}\big\| \sum_{n\geq 1} \epsilon_n dx_n \big\|_{E(\M)}
\]
where $(\epsilon_n)$ denotes a Rademacher sequence on a given probability space.  According to \cite[Theorem~4.3]{Dirk-Pag-Pot-Suk}, we then have,
\[
 (\kappa_E')^{-1} \big\|(dx_n)\big\|_{E(\M; \ell_2^c) + E(\M; \ell_2^r)} \leq \big\| x\big\|_{E(\M)} \leq c_E  \big\|(dx_n)\big\|_{E(\M; \ell_2^c) \cap E(\M; \ell_2^r)}.
\]
Using the   the noncommutative Stein inequality  on the first inequality (\cite{Jiao2}), we deduce that
\[
C_E^{-1}  \big\|x\big\|_{\H_E^c + \H_E^r} \leq \big\| x\big\|_{E(\M)} \leq  c_E\big\|x\big\|_{\H_E^c \cap \H_E^r}.
\]
This proves both items.
\end{proof}
\begin{proof}[Proof of Theorem~\ref{appl-Burk2}]
Let $ x\in E(\M)$ and $\epsilon>0$. By Proposition~\ref{BG}(i), there exists a decomposition 
$x=a^c + a^r$ so that 
\begin{equation}
\| a^c\|_{\H_E^c} + \| a^r\|_{\H_E^r} \leq  C_E \|x\|_{E(\M)} +\varepsilon.
\end{equation}
Applying Theorem~\ref{E-davis} separately on  $a^c$ and $(a^r)^*$, there exist further decompositions $a^c=y^d + y^c$ and $a^r= z^d + z^r$ with $y^d, z^d \in \h_E^d(\M)$, $y^c \in \h_E^c(\M)$, and $z^r \in \h_E^r(\M)$ satisfying:
\[
\|y^d\|_{\h_E^d} + \|y^c\|_{\h_E^c}\leq \delta_E \|a^c\|_{\H_E^c} +\epsilon \ \ \text{and}\ \ 
\|z^d\|_{\h_E^d} + \|z^r\|_{\h_E^r}\leq \delta_E \|a^r\|_{\H_E^r} +\epsilon.
\]
Set $x^d:= y^d +z^d$, $x^c:=y^c$, and $x^r:=z^r$. Clearly, $x=x^d + x^c + x^r$ and the previous two inequalities lead to:
\[
\big\|x^d \big\|_{\h_E^d} +\big\|x^c \big\|_{\h_E^c} +\big\|x^r \big\|_{\h_E^r} \leq \delta_EC_E \|x\|_{E(\M)} + (\delta_E +2)\epsilon.
\]
This proves (i). Item (ii) is similar.  Indeed, from  combining Theorem~\ref{E-davis}(ii) and Proposition~\ref{BG}(ii), we deduce  that for every finite martingale $x$, one has:
\begin{align*}
\|x\|_{E(\M)} &\leq c_E \max\{ \|x\|_{\H_E^c(\M)}, \|x\|_{\H_E^r(\M)}\}\\
&\leq c_E \beta_E  \max\{ \|x\|_{\h_E^d(\M)}, \|x\|_{\h_E^c(\M)}, \|x\|_{\h_E^r(\M)}\}.
\end{align*}
The proof is complete.
\end{proof}

\begin{remark} In order to have equivalences of norms as stated in Theorem~\ref{appl-Burk}, 
the assumptions used there  are in general necessary.  Indeed, 
if $E$ is a symmetric Banach function space  that satisfies the equivalences of norms  as stated in Theorem~\ref{appl-Burk}  then a fortiori,  martingale difference sequences are unconditional in $E(\M)$. From \cite{LT}, it follows that there exist $1<p  \leq q <\infty$ so that $E\in {\rm Int}(L_p,L_q)$. On the other hand, it was noted in  \cite{LeM-Suk} that if $1<p<2<q<\infty$ then $L_p(\M) \cap L_q(\M)$ fails to satisfy the noncommutative    Khintchine  inequalities. In particular, it must fail the equivalences of norms stated in Theorem~\ref{appl-Burk}. This shows that separating the two cases $E \in {\rm Int}(L_p,L_2)$ for $1<p<2$ and $E\in {\rm Int}(L_2,L_q)$ for $2<q<\infty$ are necessary. On the other hand, there are symmetric function spaces with Boyd indices equal to $2$ but do not appear to belong to either of the two classes of functions considered in Theorem~\ref{appl-Burk}. For instance, we do not know if either of the versions of the noncommutative Burkholder/Rosenthal inequalities  in Theorem~\ref{appl-Burk} apply to martingales in  $L_{2,\infty}(\M,\T)$, or more generally in  $L_{2,q}(\M,\T)$ for any $1\le q\neq2\le\infty$.
\end{remark}

%%%%%%%%%%%%%%%%%%%%%%%%%%%%%%%%%%%%%%%%%%%%%%%%%%%%%%%%%%%%%%%%%%%%%%%
\subsection{Modular inequalities}
%%%%%%%%%%%%%%%%%%%%%%%%%%%%%%%%%%%%%%%%%%%%%%%%%%%%%%%%%%%%%%%%%%%%%%%

In this subsection, we focus on noncommutative moment inequalities associated with Orlicz functions, which were considered in \cite{Bekjan-Chen, Dirksen-Ricard, RW2}.
We will assume throughout that 
 $\Phi$ is an Orlicz function satisfying the $\Delta_2$-condition, that is, for some constant $C>0$,
\begin{equation}
\Phi(2t) \leq C \Phi(t) \quad  t\geq 0.
\end{equation}
We denote by $L_\Phi$ the Orlicz function space associated to $\Phi$. 
Below, we write $\H_\Phi^c(\M)$, $\h_\Phi^c(\M)$, {\it etc}. for  martingale Hardy spaces $\H_{L_\Phi}^c(\M)$, $\h_{L_\Phi}^c(\M)$, {\it etc}.  We make the observation that if  
$L_\Phi \nsubseteq L_2 + L_\infty$, then for an   $x \in \h_\Phi^c(\M)$, 
  the $\Phi$-moment $\T\big[\Phi(s_c(x))\big]$ is understood  to be  the quantity $\T \otimes \tr \big[ \Phi(|U\mathcal{D}_c(x)|)\big]$ as fully detailed  in \cite{RW2}.

Given $1\leq p\leq q < \infty$,  we recall that an Orlicz function $\Phi$  is said to be \emph{$p$-convex} if the function $t \mapsto\Phi(t^{1/p})$ is convex, and to be  \emph{$q$-concave} if the function $t\mapsto \Phi(t^{1/q})$ is concave. The function $\Phi$ satisfies the $\Delta_2$-condition if and only if  it is $q$-concave for some $q<\infty$. Recall the so-called  Matuzewska-Orlicz indices $p_\Phi$ and $q_\Phi$ of $\Phi$: 
  \[
  p_\Phi=\lim_{t\to 0^+}\frac{\log M_\Phi(t)}{\log t}\;\text{ and }\; q_\Phi=\lim_{t\to\infty}\frac{\log M_\Phi(t)}{\log t}\,,
  \]
 where 
  \[M_\Phi(t)=\sup_{s>0}\frac{\Phi(ts)}{\Phi(s)}\,.\]
The indices $p_\Phi$ and $q_\Phi$ are used in the previous papers \cite{Bekjan-Chen, Bekjan-Chen-Ose, Dirksen-Ricard, RW2} instead of the convexity and concavity indices in the present one. It is easy to see that $p\le p_\Phi\le q_\Phi\le q$ if $\Phi$ is $p$-convex and $q$-concave. We refer to \cite{Maligranda2}   for backgrounds on  Orlicz functions and spaces. 
 %%%

As part of our motivation, we state   the following $\Phi$-moment  version of the noncommutative Burkholder-Gundy inequality:

\begin{theorem}[{\cite{Bekjan-Chen,Dirksen-Ricard}}]\label{BG-Phi}
Let $1<p<q<\infty$ and 
$\Phi$  be a $p$-convex and $q$-concave Orlicz function. Then there exists a positive constant $c_\Phi$ such that for every $x\in L_\Phi(\M)$,
\[
c_\Phi^{-1}  \inf\Big\{ \T\big[ \Phi( S_c( y))\big] +  \T\big[ \Phi( S_r( z))\big] \Big\}\leq \T\big[\Phi(|x|)\big] \leq c_\Phi \max\Big\{  \T\big[ \Phi( S_c( x))\big],  \T\big[ \Phi( S_r( x))\big]\Big\},
\]
where the infimum on the first inequality is taken over all $y \in \H_\Phi^c(\M)$ and $z\in \H_\Phi^r(\M)$ such that $x=y+z$.
\end{theorem}

\begin{proof} The second inequality is from \cite[Corollary~3.3]{Dirksen-Ricard}. The first one follows from a $\Phi$-moment Khintchine inequality proved in \cite{Bekjan-Chen} which states that for any given finite sequence $(a_k)$ in $L_\Phi(\M)$,
\begin{equation}\label{Kh-Phi}
\inf\Big\{ \T\Big[ \Phi\big( \big( \sum_k |b_k|^2)^{1/2}\big)\big)\Big] +  \T\Big[ \Phi\big( \big( \sum_k |c_k^*|^2)^{1/2}\big)\big)\Big] \Big\}\leq C\mathbb{E}\Big[
\T\big[\Phi\big( \big| \sum_k \epsilon_k  a_k \big|\big)\big]  \Big],
\end{equation}
where $(\epsilon_k)_{k\geq 1}$ is a Rademacher sequence and the infimum runs over all  decompositions $a_k=b_k +c_k$ with $b_k$ and $c_k$ in $L_\Phi(\M)$. We should note that \eqref{Kh-Phi} was stated in \cite{Bekjan-Chen} under the assumption that $1<p_\Phi \leq q_\Phi <2$ but the proof given there apply verbatim to the  present situation. It is now standard to deduce the first inequality from \eqref{Kh-Phi} using the $\Phi$-moment versions  of  the noncommutative Stein inequality and  martingale transforms. Both of these  results were proved in \cite{Bekjan-Chen}.
\end{proof}
It is a natural question if  the Burkholder/Rosenthal version of the above theorem holds. 
A first attempt in this direction was done in \cite{RW2} but the results obtained there  require far more restrictive assumption than the one in Theorem~\ref{BG-Phi}. As in the case of noncommutative symmetric spaces, our approach is based on the consideration of our Davis decomposition. 
The following is one of our  main results in this subsection. It  is  the $\Phi$-moment analogue of  the Davis decomposition stated in
Theorem~\ref{E-davis}.
\begin{theorem}\label{Phi-davis} Let $1<p<q<\infty$ and 
$\Phi$  be a $p$-convex and $q$-concave Orlicz function. Then there exist positive constants $\a_\Phi$ and $\b_\Phi$ such that:
\begin{enumerate}[{\rm(i)}]
 \item for every martingale $x \in \H_\Phi^c(\M)$, the following inequality holds:
 \begin{equation*}
\a_\Phi^{-1}  \inf\Big\{ \T\big[ \Phi( s_c( x^c))\big] +   \sum_{n\geq 1} \T\big[\Phi(|dx_n^d|)\big] \Big\}\leq \T\big[\Phi(S_c(x))\big], 
\end{equation*}
where  the infimum   is  taken over all  $x^c \in \h_\Phi^c(\M)$ and $x^d \in \h_\Phi^d(\M)$ such that $x= x^c  + x^d$;
\item for every $x \in \h_\Phi^d(\M) \cap \h_\Phi^c(\M)$, the following inequality holds:
\begin{equation*}
\T\big[\Phi(S_c(x))\big] \leq \b_\Phi \max\Big\{ \sum_{n\geq 1}\T\big[ \Phi\big( |dx_n|\big) \big], \T\big[ \Phi( s_c( x))\big]\Big\}.
\end{equation*}
\end{enumerate}
\end{theorem}

%%%%
Before we present the proof,  we need some preparations. We  first record  few technical facts  from interpolation and duality that we will need in the sequel. 

\begin{lemma}[{\cite[Lemma~6.2]{Jiao-Sukochev-Zanin}}]\label{interpol-lemma}
Let $(\M_1,\T_1)$ and $(\M_2, \T_2)$ be semifinite von Neumann algebras and $\Phi$ be a $p$-convex and $q$-concave Orlicz function for  $1\leq p\leq q<\infty$. If $W: L_p(\M_1) \to L_p(\M_2)$ and $W: L_q(\M_1) \to L_q(\M_2)$ are bounded linear operators, then  there exists a constant $C_\Phi$  satisfying:
\[
\T_2\big[ \Phi(|Wx|)\big] \leq C_\Phi \T_1\big[ \Phi(|x|)\big], \quad x \in L_\Phi(\M_1).
\]
\end{lemma}

Lemma~\ref{interpol-lemma}  shows in particular that if $\Phi$ is $p$-convex and $q$-concave then the Orlicz function space $L_\Phi$ belongs to ${\rm Int}(L_p,L_q)$.

\begin{lemma}[{\cite{RW2})}]\label{basic:ineq} Let $\N$ be a semifinite von Nemmann algebra and $\Phi$ be an Orlicz function such that $1<p<p_\Phi\leq q_\Phi <q <\infty$. The following inequalities hold:
\begin{enumerate}[{\rm(i)}]
\item For every $y \in L_p(\N) + L_q(\N)$,
\[
\int_0^\infty \Phi\big[t^{-1}K\big(t,y; L_1(\N), \N \big)\big]\ dt \leq C_{\Phi,p,q}
\int_0^\infty \Phi\big[t^{-1/p}K\big(t^{1/p-1/q},y; L_p(\N),  L_q(\N) \big)\big]\ dt.\]
\item If $y \in L_p(\N) \cap L_q(\N)$ and $u(\cdot)$ is a representation of $y$ in the couples $(L_p(\N), L_q(\N))$ and $(L_1(\N), \N)$ then,
\[\int_0^\infty \Phi\big[t^{-1/p}J\big(t^{1/p-1/q},u(t); L_p(\N),  L_q(\N) \big)\big]\ dt
\leq C_{\Phi,p,q} \int_0^\infty \Phi\big[t^{-1}J\big(t,u(t); L_1(\N), \N \big)\big]\ dt.\]
\item For every $y \in L_p(\N) + L_q(\N)$, 
\[
\int_0^\infty \Phi\big[t^{-1/p}K\big(t^{1/p-1/q},y; L_p(\N),  L_q(\N) \big)\big]\ dt \leq C_{\Phi,p,q} \inf \Big\{\int_0^\infty \Phi\big[t^{-1/p}J\big(t^{1/p-1/q},u(t); L_p(\N),  L_q(\N) \big)\big]\ dt\Big\},
\]
where the infimum is taken over all representations $u(\cdot)$ of $y$.
\end{enumerate}
\end{lemma}
%%%

Below  $\Phi^*$  denotes the Orlicz complementary function to $\Phi$. The next lemma will be used for duality purposes.
\begin{lemma}[{\cite[Proposition~2.3]{RW2}}]\label{duality}
Let $\Phi$ be an Orlicz function which  is $p$-convex and $q$-concave for some $1<p\leq q<\infty$ and $\N$ be a semifinite von Neumnn.  For every $0\leq x \in L_\Phi(\N)$ there exists $0\leq y \in L_{\Phi^*}(\N)$ such that $y$ commutes with $x$  and satisfies:
 $xy=\Phi(x) +\Phi^*(y)$.
\end{lemma}

\begin{proof}[Proof of  Theorem~\ref{Phi-davis}]
The proof is an adaptation of the argument used in \cite{RW2}  so we will only highlight the main points. We begin with  the proof of (i). Since $\Phi$ is $p$-convex and $q$-concave, we have $p\leq p_\Phi \leq q_\Phi\leq q$. Let $1<p_0<p$ and $q<q_0<\infty$. It is clear that $\Phi$ is $p_0$-convex and $q_0$-concave. Replacing $p$ by $p_0$ and $q$ by $q_0$ if necessary, we may assume without loss of generality that 
$1<p<p_\Phi\leq q_\Phi<q<\infty$. Under this assumption, Lemma~\ref{basic:ineq} applies to $\Phi$. Below, $C_{\Phi,p,q}$  denotes a constant whose value may change from line to line.

Fix  a martingale $x$  such that  $\xi =(dx_n)_{n\geq 1} \in L_1(\M;\ell_2^c) \cap L_\infty(\M; \ell_2^c)$. We make the observation that by complementation,
\[
K(t,\xi)=K\big(t,\xi; L_1(\M;\ell_2^c), L_\infty(\M; \ell_2^c)\big)=\int_0^t \mu_s(S_c(x))\ ds, \quad t>0.
\]
This leads to 
\begin{equation}\label{Obs}
\T\big[ \Phi\big( S_c(x) \big) \big] =\int_0^\infty \Phi(\mu_t(S_c(x)))\ dt \approx \int_0^\infty  \Phi\Big[ t^{-1} K(t, \xi)\Big]\ dt
\end{equation}
where the equivalence comes from the boundedness of the Hilbert operator on $L_r$ for $1<r<\infty$ and Lemma~\ref{interpol-lemma}.

Choose  $u(\cdot)$ a representation of  $\xi$ in the compatible couple $(L_1(\M;\ell_2^c) , L_\infty(\M;\ell_2^c))$ such that:

\begin{equation}\label{J-K}
J(t, u(t)) \leq C K(t,\xi), \quad t>0,
\end{equation}
where $C$ is an absolute constant.
Thus, since $\Phi$ has the $\Delta_2$-condition, we have from \eqref{Obs} and  \eqref{J-K}  that
\begin{equation}\label{J-L1}
\int_0^\infty  \Phi\Big[ t^{-1} J\big(t, u(t)\big)\Big]\ dt  \leq C_\Phi \T\big[ \Phi\big(S_c(x)\big) \big].
\end{equation}
It is important to note that $u(\cdot)$ is also a representation of $\xi$  for the  couple $(L_p(\M;\ell_2^c), L_q(\M; \ell_2^c))$. 
 Putting   \eqref{J-L1} together with Proposition~\ref{basic:ineq}(ii)  yields:
\begin{equation*}
\int_0^\infty  \Phi\Big[ t^{-1/p} J\big(t^{1/p-1/q}, u(t); L_p(\M;\ell_2^c), L_q(\M;\ell_2^c)\big)\Big]\ dt  \leq C_{\Phi ,p,q}\T\big[ \Phi\big(S_c(x)\big) \big].
\end{equation*}
Consider $\Theta: L_p(\M;\ell_2^c) + L_q(\M;\ell_2^c) \to \H_p^c(\M) +\H_q^c(\M)$ defined by: 
\[
\Theta\big( (a_n)_{n\geq 1}\big) = \sum_{n\geq 1} \big[\E_n(a_n)-\E_{n-1}(a_n)\big].
\]
By the noncommutative Stein inequality, $\Theta$ is bounded and one can easily verify that $\Theta\big( u(\cdot)\big)$ is a representation of $x$ for the couple $(\H_p^c(\M), \H_q^c(\M))$. Moreover, we have:
\begin{equation*}
\int_0^\infty  \Phi\Big[ t^{-1/p} J\big(t^{1/p-1/q}, \Theta\big(u(t)\big); \H_p^c(\M), \H_q^c(\M)\big)\Big]\ dt  \leq C_{\Phi ,p,q}\T\big[ \Phi\big(S_c(x)\big) \big].
\end{equation*}
As in \cite{RW2}, we
 need to modify the representation as follows: set $\theta=1/p -1/q$ and define:
\[v(t)=\frac1\theta\,\Theta\big(u(t^{1/\theta}a)\big).
\]
Then  $v(\cdot)$ is a representation of $x$ in the couple $(\H_p^c(\M), H_q^c(\M))$ and 
the preceding inequality becomes:
\begin{equation}\label{J-p-v}
\int_0^\infty  \Phi\Big[ t^{-1/p} J\big(t^{1/p-1/q}, v(t^{1/p-1/q}); \H_p^c(\M), \H_q^c(\M)\big)\Big]\ dt  \leq C_{\Phi ,p,q}\T\big[ \Phi\big(S_c(x)\big) \big].
\end{equation}
 Next, we discretize the integral in \eqref{J-p-v}. If we set $v_\nu=\int_{2^\nu}^{2^{\nu +1}} v(t)\ dt/t$ for every $\nu \in \mathbb{Z}$,  then  
 \begin{equation}\label{discrete-v}
 x=\sum_{\nu \in \mathbb{Z}} v_\nu,\; \ \text{convergence in }\; \H_p^c(\M) + \H_q^c(\M).
 \end{equation}
By \cite[Lemma~3.12(i)]{RW2}, we  deduce from \eqref{J-p-v} that 
\begin{equation}\label{J-p-v-discrete}
\sum_{\nu \in \mathbb{Z}} 2^{\nu/\theta}  \Phi\Big[2^{-\nu/(\theta p)} J\big(2^\nu, v_\nu; \H_p^c(\M), \H_q^c(\M)\big) \Big]  \leq C_{\Phi ,p,q}\T\big[ \Phi\big(S_c(x)\big) \big].
\end{equation}
 
\medskip

The next step is to apply the simultaneous Davis decomposition 
(Corollary~\ref{davis:1} and Remark~\ref{davis:simul}). For each $\nu \in \mathbb{Z}$,  there exist $a_\nu \in \h_p^d(\M)\cap \h_q^d(\M)$ and  $b_\nu \in \h_p^c(\M) \cap \h_q^c(\M)$ such that:
\begin{equation}\label{decompositions}
v_\nu = a_\nu + b_\nu 
\end{equation}
and if $s$ is equal to either $p$ or $q$,  then
\begin{equation}\label{double-decomposition}
\big\|a_\nu \big\|_{\h_s^d} + \big\|b_\nu \big\|_{\h_s^c}  \leq 
C(p,q) \big\| v_\nu \big\|_{\H_s^c}.
\end{equation}
As above, the inequalities in  \eqref{double-decomposition} can be reinterpreted using the $J$-functionals as follows:
\begin{equation}\label{J-inequalities}
\begin{split}
J\big( t, a_\nu; \h_p^d(\M), \h_q^d(\M)\big) &\leq C(p,q) J\big( t, v_\nu; \H_p^c(\M), \H_q^c(\M)\big),\  
t>0,\\
 J\big(t, b_\nu; \h_p^c(\M), \h_q^c(\M)\big) &\leq C(p,q) J\big( t, v_\nu; \H_p^c(\M), \H_q^c(\M)\big), \ t>0.
\end{split}
\end{equation}
Following similar argument  used in the proof of \cite[Sublemma~3.3]{RW}, we  have that the series $ \sum_{\nu \in \mathbb{Z}} a_\nu$ and $ \sum_{\nu \in \mathbb{Z}} b_\nu$  are convergent in $\h_\Phi^d(\M)$ and $\h_\Phi^c(\M)$ respectively. Set
\begin{equation}\label{sum}
a:= \sum_{\nu \in \mathbb{Z}} a_\nu \in \h_\Phi^d(\M)\ \text{and}\ 
b:=\sum_{\nu \in \mathbb{Z}} b_\nu \in \h_\Phi^c(\M).
\end{equation}
Combining \eqref{J-p-v-discrete} with \eqref{J-inequalities} we further get:
\begin{equation}\label{J2-discrete}
\begin{split}
\sum_{\nu \in \mathbb{Z}} 2^{\nu/\theta}  \Phi\Big[2^{-\nu/(\theta p)} J(2^\nu, a_\nu; \h_p^d(\M), \h_q^d(\M)) \Big]  &\leq C_{\Phi ,p,q}\T\big[ \Phi\big(S_c(x)\big) \big],\\
\sum_{\nu \in \mathbb{Z}} 2^{\nu/\theta}  \Phi\Big[2^{-\nu/(\theta p)} J(2^\nu, b_\nu; \h_p^c(\M), \h_q^c(\M)) \Big]  &\leq C_{\Phi ,p,q}\T\big[ \Phi\big(S_c(x)\big) \big].\end{split}
\end{equation}
Next, we convert the above inequalities  into their  corresponding continuous forms. By setting for $t\in [2^\nu, 2^{\nu +1})$,
\[
a(t)=\frac{a_\nu}{\log2} \in \h_p^d(\M) \cap \h_q^d(\M) \ \text{and} \ b(t)=\frac{b_\nu}{\log2} \in \h_p^c(\M) \cap \h_q^c(\M),
\]
we get that $a(\cdot)$ is a representation of $a$ for the couple $(\h_p^d(\M),\h_q^d(\M))$ and $b(\cdot)$ is a representation of $b$ for the couple $(\h_p^c(\M),\h_q^c(\M))$. Moreover, \cite[Lemma~3.12(ii)]{RW2} and \eqref{J2-discrete} provide  integral estimates involving the $J$-functionals:
\begin{equation}\label{J2-continuous}
\begin{split}
\int_0^\infty \Phi\Big[ t^{-1/p} J\big(t^{1/p-1/q}, a(t^{1/p-1/q}); \h_p^d(\M), \h_q^d(\M) \big) \Big]\ dt &\leq C_{\Phi ,p,q}\T\big[ \Phi\big(S_c(x)\big) \big],\\
\int_0^\infty \Phi\Big[ t^{-1/p} J\big(t^{1/p-1/q}, b(t^{1/p-1/q}); \h_p^c(\M), \h_q^c(\M) \big) \Big]\ dt &\leq C_{\Phi ,p,q}\T\big[ \Phi\big(S_c(x)\big) \big].
\end{split}
\end{equation}
By Lemma~\ref{basic:ineq}(iii), these further yield:
\begin{equation}\label{K2-continuous}
\begin{split}
\int_0^\infty \Phi\Big[ t^{-1/p} K\big(t^{1/p-1/q}, a; \h_p^d(\M), \h_q^d(\M) \big) \Big]\ dt &\leq C_{\Phi ,p,q}\T\big[ \Phi\big(S_c(x)\big) \big],\\
\int_0^\infty \Phi\Big[ t^{-1/p} K\big(t^{1/p-1/q}, b; \h_p^c(\M), \h_q^c(\M) \big) \Big]\ dt &\leq C_{\Phi ,p,q}\T\big[ \Phi\big(S_c(x)\big) \big].
\end{split}
\end{equation}

Let $\N_1 :=\M\overline{\otimes} \ell_\infty$ and $\N_2 := \M\overline{\otimes} B(\ell_2(\mathbb{N}^2))$. Since  for every $1\leq r<\infty$, $\mathcal{D}_d: \h_r^d(\M) \to L_r(\N_1)$ and 
$U\mathcal{D}_c: \h_r^c(\M) \to L_r(\N_2)$ are isometries,  inequalities \eqref{K2-continuous} implies:
\begin{equation}\label{K3-continuous}
\begin{split}
\int_0^\infty \Phi\Big[ t^{-1/p} K\big(t^{1/p-1/q}, \mathcal{D}_d(a); L_p(\N_1), L_q(\N_1) \big) \Big]\ dt &\leq C_{\Phi ,p,q}\T\big[ \Phi\big(S_c(x)\big) \big],\\
\int_0^\infty \Phi\Big[ t^{-1/p} K\big(t^{1/p-1/q}, U\mathcal{D}_c(b); L_p(\N_2), L_q(\N_2) \big) \Big]\ dt &\leq C_{\Phi ,p,q}\T\big[ \Phi\big(S_c(x)\big) \big].
\end{split}
\end{equation}
We apply Lemma~\ref{basic:ineq}(i)  to see that  the next two inequalities follow from \eqref{K3-continuous}:  
\begin{equation}\label{last}
\begin{split}
\int_0^\infty \Phi\Big[ t^{-1} K\big(t, \mathcal{D}_d(a) ; L_1(\N_1), \N_1 \big) \Big]\ dt &\leq C_{\Phi ,p,q}\T\big[ \Phi\big(S_c(x)\big) \big],\\
\int_0^\infty \Phi\Big[ t^{-1} K\big(t, U\mathcal{D}_c(b); L_1(\N_2),\N_2 \big) \Big]\ dt &\leq C_{\Phi ,p,q}\T\big[ \Phi\big(S_c(x)\big) \big].
\end{split}
\end{equation}
To conclude the proof, we observe  that $\int_0^\infty \Phi\Big[ t^{-1} K\big(t, \mathcal{D}_d(a) ; L_1(\N_1), \N_1 \big) \Big]\ dt \approx_\Phi \T_1\big[\Phi\big(|\mathcal{D}_d(a)|\big)\big]$ and $\int_0^\infty \Phi\Big[ t^{-1} K\big(t, U\mathcal{D}_c(b); L_1(\N_2),\N_2 \big) \Big]\ dt \approx_\Phi \T_2\big[\Phi\big(|U\mathcal{D}_c(b)|\big)\big]$ where $\T_1$ and $\T_2$ are the natural traces on $\N_1$ and $\N_2$, respectively. It is now straightforward to verify that $\T_1\big[\Phi\big(|\mathcal{D}_d(a)|\big)\big]=\sum_{n\geq 1} \T\big[\Phi\big(|da_n|\big)\big]$ and $\T_2\big[\Phi\big(|U\mathcal{D}_c(b)|\big)\big]=\T\big[\Phi\big(s_c(b)\big)\big]$, that is, we obtain that $x=a+b$ and $\sum_{n\geq 1} \T\big[\Phi\big(|da_n|\big)\big] + \T\big[\Phi\big(s_c(b)\big)\big] \leq C_\Phi \T\big[ \Phi\big(S_c(x)\big) \big]$. The proof of (i) is complete. 

\medskip

Now we provide the argument for (ii). We adapt the duality technique used in \cite{RW2}. Assume that $\Phi$ is $p$-convex and $q$-concave.  If $1/p +1/{p'}=1$ and $1/q +1/{q'}=1$, then $1<q'<p'<\infty$. We observe that 
$\Phi^*$ is $q'$-convex and $p'$-concave.  Therefore the inequality stated in  (i) applies to $\Phi^*$. By approximation, it is enough to verify the inequality for $x\in L_1(\M)\cap \M$. Let $\N=\M \overline{\otimes} B(\ell_2)$ equipped with its natural trace which we will denote by $\T_\N$. 

For $1<r<\infty$,  consider  $\Pi: L_r(\N) \to L_r(\N)$ defined by setting $\Pi\big( (a_{ij}) \big)=\sum_{n\geq 1} a_{1n} \otimes e_{n,1}$. Clearly, $\Pi$ is a contraction. Using the noncommutative Stein inequality, $\widetilde{\Pi}: L_r(\N) \to L_r(\N)$ given by 
$\widetilde{\Pi}\big( (a_{ij}) \big)=\sum_{n\geq 1} [\E_n(a_{1n}) -\E_{n-1}(a_{1n})] \otimes e_{n,1}$ is also bounded for all $1<r<\infty$. By Lemma~\ref{interpol-lemma}, there exists a constant $C_{\Phi^*}$ so that
\begin{equation}\label{projection}
\T_{\N}\big[ \Phi^*\big(|\widetilde{\Pi}((a_{ij}))| \big)\big] \leq C_{\Phi^*} \T_{\N}\big[ \Phi^*\big(|(a_{ij})| \big)\big].
\end{equation}

As in \cite{RW2},  we may fix $t_\Phi>0$ so that for every  operator $0\leq z \in L_{\Phi^*}(\N)$,
\begin{equation}\label{choose-t}
\Phi^*(t_\Phi z) \leq (2C_{\Phi^*}\alpha_{\Phi^*})^{-1} \Phi^*(z)
\end{equation}
where $\alpha_{\Phi^*}$ is the constant from (i) applied to $\Phi^*$.

Set $w :=\sum_{n\geq 1} dx_n \otimes e_{n,1} \in L_1(\N) \cap \N$.
It is clear  that $|w|=S_c(x) \otimes e_{1,1}$. 
 By Lemma~\ref{duality},   we may choose $0\leq y \in L_{\Phi^*}(\N)$ such that  $y$ commutes with $|w|$ and  
\begin{equation}\label{choose-y}
 \Phi\big(|w|\big) + \Phi^*(y) =y|w|.
\end{equation}
If $w=u|w|$ is  the polar decomposition of $w$,  we set $z :=yu^*\in L_1(\N)\cap \N$.
Write $z=(z_{ij})$ where $z_{ij} \in L_1(\M) \cap \M$ and set
\[
v = \sum_{n\geq 1} \E_n(z_{1n})-\E_{n-1}(z_{1n}).
\]
It is clear that   $\T_{\N}(y|w|)=\T_{\N}(zw)=\T(vx)$ and $v \in L_1(\M) \cap \M$. In particular, we may view $v$ as a martingale in $\H_{\Phi^*}^c(\M)$.
We now apply (i)  to  $v$.  There exists a decomposition $v=v^d  + v^c$
with $v^c \in \h_{\Phi^*}^c(\M)$  and $v^d \in \h_{\Phi^*}^d(\M)$ satisfying:
\begin{equation}\label{decomposition-y}
\T\big[ \Phi^*( s_c( v^c))\big]   + \sum_{n\geq 1} \T\big[\Phi^*(|dv_n^d|)\big]  \leq 2\alpha_{\Phi^*}  \T\big[\Phi^*\big( S_c(v)\big)\big].
\end{equation}
 Taking traces on \eqref{choose-y} together with the decomposition of $v$,  we have
 \begin{align*}
 \T_{\N}\big[\Phi(|w|)\big]  +\T_{\N}\big[ \Phi^*(y)] &=  \T(xv)\\
&=\T(xv^d) + \T(xv^c).
 \end{align*}
 One can verify that the following estimates hold (see \cite{RW2} for details):
\begin{equation*}
\T(xv^d)
\leq  \sum_{n\geq 1}\T\big[ \Phi\big(t_\Phi^{-1}|dx_n|\big)\big] + (2C_{\Phi^*}\alpha_{\Phi^*})^{-1} \sum_{n\geq 1}\T\big[\Phi^*\big( |dv_n^d|\big)\big]
\end{equation*}
and
\[  \T(xv^c) \leq \T\big[ \Phi\big(t_\Phi^{-1} s_c(x) \big)\big]  + (2C_{\Phi^*}\alpha_{\Phi^*})^{-1} \T\big[ \Phi^*\big(s_c(v^c)\big) \big].\]
 Applying  \eqref{projection}, \eqref{decomposition-y}, and taking the sum  of the previous two estimates, we arrive at:
\begin{equation*}
\begin{split}
 \T_{\N}\big[\Phi(|w|)\big]  + \T_{\N}\big[\Phi^*(y) \big]  &\leq  \sum_{n\geq 1}\T\big[ \Phi\big(t_\Phi^{-1}|dx_n|\big)\big]  + \T\big[ \Phi\big(t_\Phi^{-1} s_c(x) \big)\big]  +  C_{\Phi^*}^{-1}\T\big[\Phi^*(S_c(v)) \big]\\
 &\leq  \sum_{n\geq 1}\T\big[ \Phi\big(t_\Phi^{-1}|dx_n|\big)\big]  + \T\big[ \Phi\big(t_\Phi^{-1} s_c(x) \big)\big]  +  C_{\Phi^*}^{-1}\T_{\N}\big[\Phi^*(|\widetilde{\Pi}(z)|) \big]\\
 &\leq \sum_{n\geq 1}\T\big[ \Phi\big(t_\Phi^{-1}|dx_n|\big)\big]  + \T\big[ \Phi\big(t_\Phi^{-1} s_c(x) \big)\big]  +  \T_{\N}\big[\Phi^*(|z|) \big].
\end{split}
\end{equation*}
As we clearly have  $ \T_{\N}\big[\Phi^*(|z|) \big] \leq  \T_{\N}\big[\Phi^*(y) \big]$, we deduce  that 
\begin{align*}
 \T\big[\Phi(S_c(x))\big]  =\T_{\N}\big[\Phi(|w|) \big] &\leq \sum_{n\geq 1}\T\big[ \Phi\big(t_\Phi^{-1}|dx_n|\big)\big]  + \T\big[ \Phi\big(t_\Phi^{-1} s_c(x) \big)\big] \\
 &\leq 2\max\Big\{\sum_{n\geq 1}\T\big[ \Phi\big(t_\Phi^{-1}|dx_n|\big)\big]  , \T\big[ \Phi\big(t_\Phi^{-1} s_c(x) \big)\big] \Big\}.
\end{align*}
The existence of the constant $\beta_\Phi$   in the statement (ii)  now follows from the $\Delta_2$-condition.
\end{proof}

The next result solves \cite[Problem~6.3]{RW2}. It can be deduced at once  from  combining both the row and column versions of Theorem~\ref{Phi-davis} with Theorem~\ref{BG-Phi}.

\begin{theorem}\label{main-Phi2} Let $1<p<q<\infty$ and 
$\Phi$  be an Orlicz function that is $p$-convex and $q$-concave. Then there exist positive constants $\delta_\Phi$ and $\eta_\Phi$ such that:
\begin{enumerate}[{\rm(i)}]
 \item for every martingale $x \in L_\Phi(\M)$, the following inequality holds:
 \begin{equation*}
\delta_\Phi^{-1}  \inf\Big\{ \T\big[ \Phi( s_c( x^c))\big] +  \T\big[ \Phi( s_r( x^r))\big]+ \sum_{n\geq 1} \T\big[\Phi(|dx_n^d|)\big] \Big\}\leq \T\big[\Phi(|x|)\big], 
\end{equation*}
with  the infimum    being taken over all  $x^c \in \h_\Phi^c(\M)$, $x^r \in \h_\Phi^r(\M)$, and $x^d \in \h_\Phi^d(\M)$ such that $x= x^d +x^c +x^r$;
\item for every $x \in \h_\Phi^d(\M) \cap \h_\Phi^c(\M)\cap \h_\Phi^r(\M)$, the following inequality holds:
\begin{equation*}
\T\big[\Phi(|x|)\big] \leq \eta_\Phi \max\Big\{ \sum_{n\geq 1}\T\big[ \Phi\big( |dx_n|\big) \big], \T\big[ \Phi( s_c( x))\big],  \T\big[ \Phi( s_r( x))\big]\Big\}.
\end{equation*}
\end{enumerate}
\end{theorem}

We are now in a position of stating our $\Phi$-moment version of Burkholder/Rosenthal inequality. It  should be compared with  a recent  version of the Burkholder-Gundy inequality from \cite[Theorem~7.2]{Jiao-Sukochev-Zanin-Zhou}. Our result is much more general than the version obtained in \cite{RW2}. In fact, it solves  \cite[Problems~6.4]{RW2}. 

\begin{theorem}\label{Phi-burk} Let 
$\Phi$  be an Orlicz function. 
\begin{enumerate}[{\rm(i)}]
 \item  If $\Phi$ is $p$-convex for some $1<p<2$ and  $2$-concave, then there exists a positive constant $C_\Phi$ so that  the following  holds for every martingale $x \in L_\Phi(\M)$:
 \begin{equation*}
 C_\Phi^{-1} 
\T\big[\Phi(|x|)\big] \leq 
 \inf\Big\{ \T\big[ \Phi( s_c( x^c))\big] +  \T\big[ \Phi( s_r( x^r))\big]+ \sum_{n\geq 1} \T\big[\Phi(|dx_n^d|)\big] \Big\} \leq C_\Phi \T\big[\Phi(|x|)\big],
\end{equation*}
where  the infimum   is  taken over all  $x^c \in \h_\Phi^c(\M)$, $x^r \in \h_\Phi^r(\M)$, and $x^d \in \h_\Phi^d(\M)$ such that $x= x^d +x^c +x^r$;
\item  If $\Phi$ is $2$-convex and $q$-concave for some $q>2$, then there exists a positive constant $c_\Phi$ so that the following  holds for every martingale $x \in L_\Phi(\M)$:
\begin{equation*}
c_\Phi^{-1}\T\big[\Phi(|x|)\big] \leq c_\Phi \max\Big\{ \sum_{n\geq 1}\T\big[ \Phi\big( |dx_n|\big) \big],\; \T\big[ \Phi( s_c( x))\big], \; \T\big[ \Phi( s_r( x))\big]\Big\} \leq  c_\Phi \T\big[\Phi(|x|)\big].
\end{equation*}
\end{enumerate}
\end{theorem}

\begin{proof}  To prove (i), it is enough to verify that  there exists a constant $C_\Phi$ so that for every 
decomposition $x=x^d + x^c + x^r$, 
\[
\T\big[\Phi(|x|)\big]\leq C_\Phi \Big\{ \T\big[ \Phi( s_c( x^c))\big] +  \T\big[ \Phi( s_r( x^r))\big]+ \sum_{n\geq 1} \T\big[\Phi(|dx_n^d|)\big] \Big\}\,,
\]
as the  reverse inequality is already  contained in Theorem~\ref{main-Phi2}.
This follows from the facts that  $\h_2^s(\M)=L_2(\M)$,  $ h_p^s(\M) \subset L_p(\M)$ for $s\in\{d,c,r\}$,  and that the   Hardy  spaces are complemented subspaces of noncommutative $L_p$-spaces. Indeed, for  $w\in\{p,2\}$, let $\Pi: 
L_w(\M \overline{\otimes} B(\ell_2(\mathbb{N}^2)))  \to \h_w^c(\M)$ be the bounded projection (see \cite{Ju} for the fact that the projections are simultaneously bounded) and 
$\Theta: \h_w^c(\M) \to L_w(\M)$ the formal inclusion. By the noncommutative Burkholder inequality,  $\Theta\circ \Pi :  L_w(\M \overline{\otimes} B(\ell_2(\mathbb{N}^2))) \to L_w(\M)$ is bounded. By  Lemma~\ref{interpol-lemma}, we have for every   $a \in \h_\Phi^c(\M)$, 
\[
\T \big[ \Phi( |a|)\big]  = \T \big[\Phi( |\Theta\Pi(U{\cal D}_c(a))| )\big]  \leq C_\Phi \T \otimes \tr \big[ \Phi(|U{\cal D}_c(a)|) \big]=C_\Phi \T\big[ \Phi(s_c(a))\big].
\]
Similar arguments can be applied to the diagonal  and the row parts.

\medskip

(ii) Assume now that  $\Phi$ is $2$-convex and $q$-concave for  some $q>2$. By the noncommutative Burkholder inequalities, the formal inclusion is bounded from $L_w(\M)$ into $\h_w^c(\M)$ for all $w\geq 2$. Denote this by $I$. We have $U\cal{D}_c I: L_w(\M) \to L_w(\M \overline{\otimes} B(\ell_2(\mathbb{N}^2)))$ is bounded for all $w\geq 2$.  We deduce from  Lemma~\ref{interpol-lemma} that  for every $b \in L_\Phi(\M)$,
\[
\T\big[\Phi(s_c(b))\big] =\T \otimes \tr \big[\Phi\big(|U\cal{D}_c I(b)|\big) \big] \leq C_\Phi \T\big[ \Phi(|b|)\big].
\]
Applying the same argument for the diagonal and the row parts, we have for every $x\in L_\Phi(\M)$,
\[
 \max\Big\{ \sum_{n\geq 1}\T\big[ \Phi\big( |dx_n|\big) \big], \T\big[ \Phi( s_c( x))\big],  \T\big[ \Phi( s_r( x))\big]\Big\} \leq C_\Phi \T\big[\Phi(|x|)\big].
\]
The reverse inequality is already contained  in Theorem~\ref{main-Phi2}(ii).
\end{proof}

%%%%%%%%%%%%%%%%%%%%%
%%%%%%%%%%%%%%%%%%%%%%%

\medskip
We conclude by  exhibiting examples of Orlicz functions  for which  the  $\Phi$-moment versions of the noncommutative Burkholder inequalities apply but not covered by the results from \cite{RW2}. 

\begin{example} Let $\Phi=t^p\log(1+t^q)$ with $p>1$ and $q>0$. One can check that $p_\Phi=p$ and $q_\Phi=p+q$. Also, since $\Phi(t)/t^p$ is increasing and $\Phi(t)/ t^{p+q}$ is decreasing, $\Phi$ is $p$-convex and $p+q$-concave.
\begin{enumerate}[{(\rm i)}]
\item If $p+q=2$ then the equivalence in Theorem~\ref{Phi-burk}(i) holds for $\Phi$.
\item If $p=2$ then the equivalence in Theorem~\ref{Phi-burk}(ii) holds for $\Phi$. 
\end{enumerate}
\end{example}

\medskip

\noindent{\bf Acknowledgments.} A portion of the work reported in this paper was completed while   the first   and   second named authors visited the Harbin Institute of Technology in the Summer of 2017. It is their pleasure to express their gratitude to the Institute for Advanced Study in Mathematics of  the HIT  for providing  stimulating working atmosphere and  for financial supports. The second named author was partially  supported by NSFC  grant  No. 11601526 and the China Postdoctoral Foundation (2016M602420, 2017T100606).
The third named author was partially supported by NSFC grant No. 11431011, the French project ISITE-BFC (ANR-15-IDEX-03) and IUF.

%%%%%%%%%%%%%%%
%\bibliography{narciref,narciref2}
%\bibliographystyle{amsplain}
%%%%%%%%%%%%%%%%%%
%%%%%%%%%%%%%%%%%%%
\def\cprime{$'$}
\providecommand{\bysame}{\leavevmode\hbox to3em{\hrulefill}\thinspace}
\providecommand{\MR}{\relax\ifhmode\unskip\space\fi MR }
\providecommand{\MRhref}[2]{%
  \href{http://www.ams.org/mathscinet-getitem?mr=#1}{#2}
}
\providecommand{\href}[2]{#2}

\end{document}